\documentclass[11pt]{amsart}%
\usepackage{amsfonts}
\usepackage{txfonts}
\usepackage{mathtools}
\usepackage{amsmath}
\usepackage{amssymb}
\usepackage{amsthm}
\usepackage{newlfont}
\usepackage{graphicx}
\RequirePackage{color}
\usepackage{color}

\providecommand{\U}[1]{\protect\rule{.1in}{.1in}}
\numberwithin{equation}{section}

\newcommand{\beq}{\begin{equation}}
\newcommand{\eeq}{\end{equation}}
\newcommand{\beqs}{\begin{eqnarray*}}
\newcommand{\eeqs}{\end{eqnarray*}}
\newcommand{\beqn}{\begin{eqnarray}}
\newcommand{\eeqn}{\end{eqnarray}}
\newcommand{\beqa}{\begin{array}}
\newcommand{\eeqa}{\end{array}}

  \newcommand{\ii}{{\sqrt{-1}}}

\newcommand{\C}{{\mathbb C}}

\newcommand{\vp}{\iota}

\newcommand{\ddbar}{\sqrt{-1} \partial \overline{\partial}}

\newtheorem{prop}{Proposition}[section]
\newtheorem{theo}[prop]{Theorem}
\newtheorem{lem}[prop]{Lemma}

\newtheorem{rem}[prop]{Remark}

\def\S{\mathbb S}
\def\C{\mathbb C}
\def\T{\mathbb T}

\def\i{\sqrt{-1}}

\subjclass[2010]{53C43, 53A10, 53C55. }


\allowdisplaybreaks
\title[Stability of harmonic maps]{On stability of the fibres of Hopf surfaces as harmonic maps and minimal surfaces}

\author{Jingyi Chen}

\address{Department of Mathematics,
The University of British Columbia,
Vancouver, BC V6T1Z2, Canada}
\email{jychen@math.ubc.ca}
\author{Liding Huang}
\address{School of Mathematical Sciences, University of Science and Technology of China, Hefei 230026, P. R. China}
\email{huangld@mail.ustc.edu.cn}

\begin{document}
\maketitle


\vspace{-.5cm}

\begin{abstract}
We construct a family of Hermitian metrics on the Hopf surface $\S^3\times\S^1$,
 whose fundamental classes represent distinct cohomology classes in the Aeppli 
 cohomology group. These metrics are locally conformally K\"ahler. Among the toric fibres of $\pi:\S^3\times\S^1\to\C P^1$ two of them are stable minimal surfaces and each of the two has a neighbourhood so that fibres therein are given by stable harmonic maps from 2-torus and outside, far away from the two tori, there are unstable harmonic ones that are also unstable minimal surfaces. A similar result is true for $\S^{2n-1}\times \S^{1}$. 
\end{abstract}

\section{Introduction}

It is interesting to know when a holomorphic map from a compact Riemann surface to a non-K\"ahlerian Hermitian manifold is area minimizing. Besides mappings between K\"ahler manifolds, little seems to be known about the second variation of holomorphic maps between Hermitian manifolds, even in the case that they are also harmonic maps, for the energy functional. A strong reason for this phenomenon is that the Riemannian properties of a Hermitian metric are considerably more complicated than those of a K\"aher metric. It demands specific information on the Riemannian curvature of the Hermitian metric to determine stability or instability from the second variation formulas. 

\vspace{.1cm}

In this paper, we construct explicitly a continuous family of Hermitian metrics on $\S^3\times\S^1$;
with these metrics, we study stability of the fibres of $\pi:\S^3\times\S^1\to\C P^1$ as harmonic maps for the energy functional and as minimal surfaces for the area functional in $\S^3\times\S^1$.

\vspace{.1cm}

A classical theorem of Kodaira \cite{KK} states that a compact complex surface homeomorphic to $\S^3\times\S^1$ is complex analytically diffeomorphic to a Hopf surface.
Such surface cannot carry any K\"ahler metric as its second Betti number is 0.
Hermitian metrics on the Hopf surfaces have been constructed explicitly; for example, Gauduchon and Ornea produced locally conformally K\"ahler metrics \cite{GO}. We demonstrate, in Section 2, a construction of Hermitian metrics ${g}_\epsilon$ via a sequence of Hermitian metrics $\tilde{g}_\epsilon$  on the Calabi-Eckmann 3-fold $\S^3\times\S^3$ induced from $\C^2\times \epsilon^2\C^2$. Each $\tilde{g}_\epsilon$ induces a Hermitian metric $g_\epsilon$ on $\S^3\times\S^1$ for $\epsilon>0$. The collapsing sequence $(\S^3\times\S^3,\tilde{g}_\epsilon)$ converges to $(\S^3\times\S^1,{g}_0)$ in the Gromov-Hausdorff distance as $\epsilon\to 0$, and ${g}_0$ is a smooth Hermitian metric on $\S^3\times\S^1$. The fundamental classes associated with these non-K\"ahler metrics belong to distinct cohomology classes in the Aeppli cohomology $H^{1,1}_A$, which is isomorphic to the Bott-Chern cohomology $H^{1,1}_{BC}$. The metrics $g_\epsilon$ are different from those in \cite{GO}, see Remark 2.7.


Throughout the paper, we use the complex structure defined by Calabi and Eckmann \cite{CE} on $\S^3\times\S^1$ and denote the complex coordinates by $w_0 = x_0+\i x_2, w_1=x_1+\i x_3$, where $w_0$ is on the fibre and $w_1$ on the base $\C P^1$, see Section 2.1.

\vspace{.1cm}

Holomorphic maps may not be harmonic in a non-K\"ahler setting, as one can multiply the target Hermitian metric by positive functions to obtain new Hermitian metrics resulting in new harmonic map equations. Lichnerowicz gave a sufficient condition in \cite{Li}; however, it is not satisfied by our ${g}_\epsilon$. Harmonicity of the holomorphic inclusions in Theorem \ref{main3} can be verified from the harmonic map equation as written by Sampson \cite{Sampson} or from the fact that $\T^2_p$ is totally geodesic in $(\S^3\times\S^1,{g}_\epsilon)$, both rely on explicit computation of geometric quantities of ${g}_\epsilon$, see Section 3. For mappings between K\"ahler manifolds, it was also observed in \cite{Li} that the energy-minimizing maps are precisely the $\overline{\partial}$-energy minimizing maps, because the energy and the $\overline{\partial}$-energy differ by a homotopy invariant, namely the K\"ahler class of the target K\"ahler manifold evaluated at the homology class of the image of the mapping (cf. \cite[p.192]{Siu-Yau}). This is not necessarily true for mappings into Hermitian manifolds.


\vspace{.1cm}

Our analysis of stability relies strongly on  the vanishing of particular curvature components of the Riemannian curvature (cf. Lemma \ref{curvature}), together with delicate control on the  positive terms (unfavorite for stability) of curvature in the second variation formula by certain part of the covariant derivatives of vector fields. This allows us to verify nonnegativity for all variation fields or exhibit negativity along certain variation fields of the Morse index form, according to whether the harmonic maps are close to or far from the two special tori (see below). 

Another interesting fact of the Riemannian curvature of $g_\epsilon$ is
$R(X,Y,\overline{X},\overline{Y})\geq 0$ for all $X,Y$ in $T^{1,0}(\S^3\times\S^1)$ and it vanishes precisely on two distinguished tori
$$\T^2_{p_0}=\{0\}\times\S^1\times\S^1,\ \ \T^2_{p_1}=\S^1\times\{0\}\times\S^1$$ in $\C^2\times\C$, corresponding to the fibres over $p_0,p_1\in\C P^1$, respectively. Vectors in the holomorphic tangent bundle $T^{1,0}(\S^3\times\S^1)$ are isotropic vectors but not vice versa. We do not know whether the curvature is nonnegative on \textit{all} isotropic 2-planes, see \cite{MM} (Compact Hermitian surfaces with nonnegative isotropic curvatures were discussed in \cite{AD}), so we shall use the real form of the second variation formula for energy. The curvature term in the second variation for area vanishes at the two minimal surfaces $\T^2_{p_0},\T^2_{p_1}$. Although it is difficult to identifying zeros of curvature in general, a link between stability and zero locus of the ambient curvature was already revealed four decades ago by Bourguignon and Yau  in \cite{BY} stating that the Riemannian curvature operator of a $K3$ surface vanishes along any nonconstant closed stable geodesic.

\vspace{.1cm}

The main results of the paper are summarized as follows. For existence of geometrically different Hermitian metrics:

\begin{theo}\label{main1}
 There is a family of Hermitian metrics ${g}_\epsilon$ on the Hopf surface $\S^3\times\S^1$ for constant $\epsilon\in[0,1]$ with the following properties:

\hspace{-.2cm}(1) Each $g_\epsilon$ is locally conformally K\"ahler and is a Gauduchon metric with non-parallel Lee form.

 \hspace{-.2cm}(2) The fundamental classes $\omega_\epsilon$ of ${g}_\epsilon$ represent different elements in $H^{1,1}_A(\S^3\times\S^1)$ for different $\epsilon$.

\hspace{-.2cm}(3) For any linearly independent $X, Y\in T^{1, 0}_x(\S^{3}\times \S^{1})$, $R\left(X, Y,\overline{X},\overline{Y}\right)\geq 0$ where $R$ is the Riemannian curvature operator and $x\in \S^{3}\times \S^{1}$, it equals to zero at $x\in \S^{3}\times \S^{1}$ if and only if $x\in \T^2_{p_0}\cup\T^2_{p_1}$.
\end{theo}

For stability of the fibres as harmonic maps and as minimal embeddings:

\begin{theo}\label{main3}
Let $\pi:\S^3\times\S^1\to\C P^1$ be the Hopf fibration and $\T^2_p=\pi^{-1}(p)$ with $p\in\C P^1$ and $p_0,p_1$ correspond to the two special tori. Then there exist neighbourhoods $U_0,U_1$ of $p_0,p_1$ in $\C P^1,$ respectively, such that the holomorphic inclusion $f_{p}:\T^2_p\hookrightarrow (\S^3\times\S^1,g_\epsilon)$ 
is a stable harmonic map when $p\in U_0\cup U_1$, and $f_{p_0},f_{p_1}$ are stable minimal embeddings. 
Furthermore, $f_p$ is an unstable harmonic map and an unstable minimal embedding, when either $x_1=0$, $|x_3|>\hspace{-.15cm}\sqrt{2}$ or $x_3=0$, $|x_1|>\hspace{-.15cm}\sqrt{2}$, 
where $w_1=x_1+\i x_3$ is the holomorphic coordinate on $\C P^1$. 
\end{theo}

Let ${\mathbf E}^{\mathbb{C}}_p=f^{*}_pT(\S^{3}\times \S^{1})\otimes \mathbb{C}$ be the complexified pullback bundle over $\T^2_p$.  There is a unique holomorphic structure $\overline{\partial }$ on ${\mathbf E}^{\mathbb{C}}_p$ (for more detail, see p.21). Let $H_{\overline{\partial}}^{0}(\T_p^{2}, \mathbf{E}_{p}^{\mathbb{C}})$ be the linear space of holomorphic sections of ${\mathbf E}^{\mathbb C}_p$.  A useful curvature property is that $R(W,\frac{\partial f_p}{\partial w_0}, \overline{W}, \overline{\frac{\partial f_p}{\partial w_0}})\geq 0$ for any smooth section $W$ of ${\mathbf E}_p^\C$, which arises in the complexified second variation formula. As an application of Theorem \ref{main3}, we have
\begin{theo}\label{main4}
Let $U_0,U_1$ be the two neighbourhoods in Theorem \ref{main3} and $p\in U_{0}\cup U_1$.  If $p\neq p_{0},p_{1}$ then $H_{\overline{\partial}}^{0}(\T_p^{2}, {\mathbf E}^{\mathbb{C}}_p)=\text{Span}_\C \left\{\frac{\partial}{\partial x_{0}}, \frac{\partial}{\partial x_{2}}\right\}$. If $p=p_{0}, p_{1}$ then $H_{\overline{\partial}}^{0}(\T_p^{2}, {\mathbf E}^{\mathbb{C}}_p)=\text{Span}_\C \left\{\frac{\partial}{\partial x_{0}}, \frac{\partial}{\partial x_{2}},\frac{\partial}{\partial x_{1}}, \frac{\partial}{\partial x_{3}}\right\}.$
\end{theo}
Theorem \ref{main4} is concluded from knowing the harmonic mappings $f_p$ are stable already. The usual practice goes in the opposite direction: estimating the Morse index from existing holomorphic (or almost holomorphic) sections, see \cite{MM}, \cite{F}, \cite{FW}. On the other hand, stability together with positivity (nonnegativity) of curvature of the ambient manifold force rigidity results for harmonic maps in many situations, as demonstrated, for example, in \cite{AL}, \cite{AM}, \cite{BBBR}, \cite{C, C1}, \cite{F}, \cite{FW}, \cite{Le}, \cite{LS}, \cite{M}, \cite{MM}, \cite{MWo}, \cite{OU}, \cite{ScY}, \cite{Siu-Yau}, \cite{Xin}, et al. We also include in Appendix a known result for $\S^3\times\S^1$ equipped with the product of standard metrics.

\vspace{.1cm}

In Section 4, Theorem \ref{main3} is generalized to $\S^{2n-1}\times\S^1$. 

\vspace{.2cm}
{\it Acknowledgements. }
This work was carried out while J. Chen was partially supported by an NSERC Discovery Grant (22R80062) and  L. Huang was visiting the Department of Mathematics at the University of British Columbia, supported by the China Scholarship Council (File No. 201906340217). Huang would like to thank UBC for the hospitality and support. 

Both authors are grateful to R. Slobodeanu for pointing out to us the closely related work \cite{SS}, in which, among other things, similar issues on stability of the fibres of Hopf surfaces with Vaisman metrics was studied. Unlike in \cite{SS}, our metrics are not Vaisman and we observe both stability and instability for the same Hopf surface.

\section{Construction of a family of Hermitian metrics on $\S^3\times \S^1$}

In this section, we construct a family of Hermitian metrics on the Hopf surface $\S^3\times\S^1$ from Hermitian metrics on the Calabi-Eckmann complex 3-manifold $\S^3\times\S^3$ via an adiabatic limit type procedure, i.e., scaling down a factor of a product. We then explore properties of these metrics and explain why they are not those constructed in \cite{GO}.

\subsection{Complex structures on $\S^{3}\times \S^{3}$ and $\S^3\times \S^1$}

We recall Calabi-Eckmann's construction of an integrable complex structure on a product of two odd dimensional unit spheres and follow the notations in \cite{CE}. We only consider $\S^{3}\times \S^{3}$ and $\S^3\times \S^1$. Let $E$ and $E'$ be the complex Euclidean space $\mathbb C^2$ and $\mathbb C^p$ respectively, where $p=1,2$, with complex affine coordinates $z_{0}, z_{1}$ on $E$ and $z'_{0},z'_1$ on $E'$ for $p=2$ and $z'_0$ for $p=1$. Let $V_{\alpha\beta}$ be the open subset of $\S^3\times \S^{2p-1}$ defined by
$$
V_{\alpha\beta}=\left\{(z_0,z_1, z_0',z'_{p-1}): (z_0,z_1)\in\S^3\subset\C^2, (z'_0,z'_{p-1})\in\S^{2p-1}\subset\C^{p}, z_{\alpha}z'_{\beta}\neq 0\right\}
$$
 where $\alpha = 0,1$ and $ \beta = 0, p-1$, with $(z'_0,z'_{p-1}):=z'_0$ for $p=1$. The family $\{V_{\alpha\beta}\}$ is an open cover of $\S^3\times \S^{2p-1}$. On $V_{\alpha\beta}$, with $\alpha, j=0,1$ and $\beta, k=0, p-1$ define 
\begin{eqnarray}
w_{\alpha j }\ &=&\ \frac{z_{j}}{z_{\alpha}},\,\,\,\,\,\,j\neq \alpha \label{1}\\
 w'_{\beta k}\ &=&\ \frac{z'_{k}}{z'_{\beta}},\,\,\,\,\,\,\,k\neq\beta\label{2}\\
t_{\alpha\beta}\ &=&\ \frac{1}{2\pi\ii}\left(\log z_{\alpha}+\sqrt{-1}\log z'_{\beta}\right) \mod(1,\ii). \label{3}
\end{eqnarray}
Let $\T^2$ be the square torus $\C/\{1,\ii\}$ with a coordinate chart defined by \eqref{3}. When $p=2$, the pair of complex numbers $(w_{\alpha j},w'_{\beta k })$ is a local inhomogeneous coordinates of $\mathbb CP^{1}\times \mathbb CP^{1}$ and $(w_{\alpha j},w'_{\beta k},t_{\alpha\beta})$ is a differentiable map from $V_{\alpha\beta}$ into $\C^2\times\T^2$; when $p=1$, \eqref{2} is vacuous and $\C P^1\times\C P^0:=\C P^1$ so $w_{\alpha j}$ is a coordinate of $\C P^1$ and $(w_{\alpha j}, t_{\alpha 0})$ maps  $V_{\alpha\beta}$ differentially into $\mathbb C^{1}\times \mathbb T^{2}$.  Furthermore, as shown in \cite[p.496]{CE}, given any  $w_{\alpha j},w'_{\beta k}\in\C$ and a point $[t_{\alpha\beta}]\in\T^2$ which is represented by a complex number $t_{\alpha\beta}$ in the congruence class of the lattice $\text{mod} \,(1,\ii)$, the quadruple $z_\alpha,z_j,z'_\beta,z'_k$ are uniquely solvable in $V_{\alpha\beta}$. In fact

\begin{prop}(Calabi-Eckmann) \label{complex coordinates}
	Each $V_{\alpha\beta}$ is homeomorphic to $\mathbb C^{p}\times \mathbb T^{2}$. On $U_{\alpha\beta}=\{(w_{\alpha j},w'_{\beta k}, t_{\alpha\beta})\in\C^{p+1}: 0<\mathfrak{Re} \,t_{\alpha\beta},\mathfrak{Im}\, t_{\alpha\beta}<1\}\subset V_{\alpha\beta}$,  $(w_{\alpha j},  w'_{\beta k}, t_{\alpha\beta})$
	is a complex coordinate system of $V_{\alpha\beta}\subset\S^3\times \S^{2p-1}$. For this complex structure, the fibre bundle $\S^3\times \S^{2p-1}\to \mathbb CP^{1}\times \mathbb CP^{p-1}$ is complex analytic and each fibre is a holomorphic nonsingular torus.
\end{prop}

\subsection{Hermitian Metrics on $\S^{3}\times \S^{3}$} 
Let $U_{\alpha\beta}$ be the coordinate neighbourhood in Proposition \ref{complex coordinates} which forms an open coordinate cover of $\S^3\times\S^3$.
Without loss of generality, we consider $U_{00}$ on which we set $$(w_0,w_1,w_2)=(t_{00}, w_{01},w'_{01})$$ for simplicity. The inclusion map $$\vp:\S^{3}\times \S^{3}\rightarrow \C^2\times \C^2$$ can be written locally in these coordinates as
\begin{equation}\label{3d map}
\left\{\begin{aligned}
z_{0}\ &=\ A^{-\frac{1}{2}}e^{\ii(\pi(w_{0}+\overline{w}_{0})+\frac{1}{2}\log B)},\\
z'_{0}\ &=\ B^{-\frac{1}{2}}e^{\pi(w_{0}-\overline{w}_{0})-\frac{1}{2}\ii\log A},\\
z_{1}\ &=\ z_{0}w_{1}=w_{1}A^{-\frac{1}{2}}e^{\ii(\pi(w_{0}+\overline{w}_{0})+\frac{1}{2}\log B)},\\
z'_{1}\ &=\ z'_{0}w_{2}=w_{2}B^{-\frac{1}{2}}e^{\pi(w_{0}-\overline{w}_{0})-\frac{1}{2}\ii\log A}
\end{aligned}\right.
\end{equation}
where
 \begin{equation}\label{A,B}
 A=1+|w_{1}|^{2}, \,\,\,\,B=1+|w_{2}|^{2}.
 \end{equation}

\vspace{.1cm}

To explicitly write down Hermitian metrics on $\S^3\times\S^3$, we compute $\partial\iota$:
\begin{equation*}
\small{
\left\{\begin{aligned}
\frac{\partial z_{0}}{\partial w_{0}}&=\ii \pi z_{0},  &\frac{\partial z'_{0}}{\partial w_{0}}&= \pi z'_{0},
&\frac{\partial z_{1}}{\partial w_{0}}&=\ii \pi z_{1},  &\frac{\partial z'_{1}}{\partial w_{0}}&= \pi z'_{1}, \\
\frac{\partial z_{0}}{\partial w_{1}}&=-\frac{1}{2}\frac{\overline{w}_{1}z_{0}}{A}, &\frac{\partial z'_{0}}{\partial w_{1}}&=-\frac{\ii}{2}\frac{\overline{w}_{1}z'_{0}}{A},&\frac{\partial z_{1}}{\partial w_{1}}&=z_{0}-\frac{1}{2}\frac{\overline{w}_{1}z_{1}}{A},\
&\frac{\partial z'_{1}}{\partial w_{1}}&=-\frac{\ii}{2}\frac{\overline{w}_{1}z'_{1}}{A},\\
\frac{\partial z_{0}}{\partial w_{2}}&=\frac{\ii}{2}\frac{\overline{w}_{2}z_{0}}{B}, &\frac{\partial z'_{0}}{\partial w_{2}}&=-\frac{1}{2}\frac{\overline{w}_{2}z'_{0}}{B}, &\frac{\partial z_{1}}{\partial w_{2}}&=\frac{\ii}{2}\frac{\overline{w}_{2}z_{1}}{B}, &\frac{\partial z'_{1}}{\partial w_{2}}&=z'_{0}-\frac{1}{2}\frac{\overline{w}_{2}z'_{1}}{B}.
\end{aligned}\right.}
\end{equation*}

\vspace{.1cm}

Now, we take Hermitian metrics on $\C^2\times \C^2$ defined by
$$
h_{\epsilon}=\displaystyle\sum_{i=0,1}dz_{i}\otimes d\overline{z}_{i}+\epsilon^{2}\displaystyle\sum_{i=0,1}dz'_{i}\otimes d\overline{z}'_{i},\,\,\,\,\epsilon\searrow 0.
$$
For any $X\in T^{1,0}(\S^3\times\S^3)$, let $\vp_*(X)^{1,0}\in T^{1,0}(\C^2\times\C^2)$ be the $(1,0)$ part of the push forward $\vp_*X$. Then
$$
\tilde{g}_\epsilon(X,Y) = h_\epsilon\left(\vp_*(X)^{1,0},\vp_*(Y)^{1,0}\right),\,\,\,\,X,Y\in T^{1,0}(\S^3\times\S^3)
$$
is a Hermitian metric on the complex manifold $\S^3\times\S^3$. Its components
$$
\tilde{g}_{\epsilon, i\overline{j}}=h_{\epsilon}\left(\vp_{*}\left(\frac{\partial}{\partial w_{i}}\right)^{1,0}, \vp_{*}\left(\frac{\partial}{\partial w_{j}}\right)^{1,0}\right)
$$ 
are given by the Hermitian matrix
\begin{equation}\label{induced metric}
\big(\tilde{g}_{\epsilon,i\overline{j}}\big)=\left(
\begin{matrix}
(1+\epsilon^{2})\pi^2& \ \ \frac{(1+\epsilon^{2})\ii \pi }{2}\frac{w_{1}}{A} & \ \  \frac{(1+\epsilon^{2})\pi}{2}\frac{w_{2}}{B}\\[3mm]
- \frac{(1+\epsilon^{2})\ii\pi}{2}\frac{\overline{w}_{1}}{A}& \ \ \frac{1}{4A}+\frac{3}{4A^2}+\frac{\epsilon^{2}}{4}\frac{|w_{1}|^2}{A^2}& \ \ \frac{-(1+\epsilon^{2})\ii}{4} \frac{\overline{w}_{1}w_{2}}{AB}\\[3mm]
\frac{(1+\epsilon^{2})\pi}{2}\frac{\overline{w}_{2}}{B}& \ \ \frac{(1+\epsilon^{2})\ii}{4} \frac{\overline{w}_{2}w_{1}}{AB} &\ \ \frac{|w_{2}|^{2}}{4B^2}+\frac{\epsilon^{2}}{4B}+\frac{3\epsilon^{2}}{4B^2}\\
\end{matrix}\right).
\end{equation}
For example,  
\begin{equation*}
\begin{split}
\vp_{*}\left(\frac{\partial}{\partial w_{0}}\right)^{1,0}=\sum_{i=0,1}\frac{\partial z_{i}}{\partial w_{0}}\frac{\partial}{\partial z_{i}}+\sum_{i=0,1}\frac{\partial z'_{i}}{\partial w_{0}}\frac{\partial}{\partial z'_{i}}
\end{split}
\end{equation*}
 and
\begin{equation*}
\begin{split}
\tilde{g}_{\epsilon}\left(\frac{\partial }{\partial w_{0}},\frac{\partial }{\partial w_{0}}\right)&=h_{\epsilon}\left(\vp_{*}\left(\frac{\partial}{\partial w_{0}}\right)^{1,0}, \vp_{*}\left(\frac{\partial}{\partial w_{0}}\right)^{1,0}\right)\\
&=\sum_{i=0,1}\left|\frac{\partial z_{i}}{\partial w_{0}}\right|^{2}+\epsilon^{2}\sum_{i=0,1}\left|\frac{\partial z'_{i}}{\partial w_{0}}\right|^{2}\\
&=\left(1+\epsilon^{2}\right)\pi^{2}.
\end{split}
\end{equation*}

For $\epsilon>0$, $\tilde{g}_\epsilon$ induces a Hermitian metric ${g}_\epsilon$ on $\S^3\times\S^1$ given in the coordinates $(w_0,w_1)$ by

\begin{equation}\label{induced metric 1}
({g}_{\epsilon,i\overline{j}})=\left(
\begin{matrix}
(1+\epsilon^{2})\pi^2& \ \ \frac{(1+\epsilon^{2})\ii\pi}{2} \frac{w_1}{A} \\[3mm]
- \frac{(1+\epsilon^{2})\ii\pi}{2}\frac{\overline{w}_{1}}{A}& \ \ \frac{1}{4A}+\frac{3}{4A^2}+\frac{\epsilon^{2}}{4}\frac{|w_{1}|^2}{A^2}\\[3mm]
\end{matrix}\right).
\end{equation}

When $\epsilon=0$ the symmetric 2-tensor $\tilde{g}_0$ is nonnegative but not positive definite; however, we will show that it still yields a Hermitian metric ${g}_0$ on $\S^3\times\S^1$.

\begin{prop}\label{submanifold}
With respect to the Calabi-Eckmann complex structure, the complex surface $\S^{3}\times \S^{1}$ is a complex submanifold of $\S^{3}\times \S^{3}$. The 2-tensor ${g}_0$, defined by ${g}_0(X,Y)=h_0\big(\vp_*(X)^{1,0},\vp_*(Y)^{1,0}\big)$ for $X,Y\in T^{1,0}(\S^3\times\S^3)$, is a Hermitian metric on $\S^{3}\times \S^{1}$.
\end{prop}
\begin{proof}
	Consider the map $\tau:\S^{3}\times \S^{1}\rightarrow \S^{3}\times \S^{3}$ given by $\tau(z_{0},z_{1}, z'_{0})=(z_{0},z_{1}, z'_{0}, 0)\in\C^2\times\C^2$ for any $(z_0,z_1,z_0')\in\C^2\times\C$ with $|z_0|^2+|z_1|^2=1$ and $|z'_0|=1$. It is obvious that $\tau(\S^{3}\times \S^{1})=(\S^{3}\times \S^{3})\cap \{z'_{1}=0\}$ and $\tau :\S^{3}\times \S^{1}\rightarrow \tau(\S^{3}\times \S^{1})$ is diffeomorphic. Moreover, $\tau$ is a holomorphic map: we may assume $z_{0}\neq 0$ without loss of generality. Then, $\tau (w_{0},\ w_{1})=(w_{0},\ w_{1}, 0)$ is holomorphic.

\vspace{.1cm}

To see ${g}_{0}$ is a Hermitian metric on $\S^{3}\times \S^{1}$, without loss of generality, we consider it in the coordinate chart $U_{00}$ of $\S^{3}\times \S^{1}$. For $\epsilon=0$

\begin{equation}
\left(\tilde{g}_{0,i\overline{j}}\right)=\left(
\begin{matrix}
\pi^2&\  \frac{\ii \pi}{2} \frac{w_1}{A} &  \ \frac{\pi}{2}\frac{w_{2}}{B}\\[3mm]
- \frac{\ii \pi}{2} \frac{\overline{w}_1}{A}&\  \ \frac{1}{4A}+\frac{3}{4A^2}&\ -\frac{\ii}{4} \frac{\overline{w}_{1}w_{2}}{AB}\\[3mm]
\frac{\pi}{2} \frac{\overline{w}_2}{B}& \ \  \frac{\ii}{4} \frac{\overline{w}_{2}w_{1}}{AB} &\ \ \frac{1}{4}\frac{|w_{2}|^{2}}{B^2}\\
\end{matrix}\right),
\end{equation}
therefore
\begin{equation*}
{g}_{0}=\displaystyle \sum_{i, j=0,1}g_{0, i\overline{j}}\,dw_{i}\otimes
d \overline{w}_j,
\end{equation*}
where
\begin{equation*}
\left({g}_{0,i\overline{j}}\right)=\left(
\begin{matrix}
\pi^2\ & \frac{\ii \pi }{2}\frac{w_{1}}{A} \\[3mm]
- \frac{\pi\ii}{2}\frac{\overline{w}_{1}}{A}\ & \ \ \frac{1}{4A}+\frac{3}{4A^2}\\[3mm]
\end{matrix}\right)
\end{equation*}
is a positive definite Hermitian matrix.
\end{proof}

\subsection{Gromov-Hausdorff convergence}
\begin{theo}\label{Gromov-Hausdorff convergence}
	The Hermitian manifolds $(\S^{3}\times \S^{3}, \tilde{g}_{\epsilon})$ converge to the Hermitian manifold $(\S^{3}\times \S^{1}, {g}_{0})$ in the Gromov-Hausdorff distance as $\epsilon\searrow 0$.
\end{theo}
\begin{proof} As previously done, we identity $\S^{3}\times \S^{1}$ with $(\S^{3}\times \S^{3})\cap \{z'_{2}=0\}$. Let ${g}_{\epsilon}$ be the pullback of $\tilde{g}_{\epsilon}$ by $\S^{3}\times \S^{1}\hookrightarrow\S^{3}\times \S^{3}$. Direct computation shows that
 $\det{g_{0}}=0$, and in $U_{00}$ the degenerating direction is
 $$
 \alpha=-\frac{\overline{w}_2}{2\pi }\,\frac{\partial}{\partial w_0}+B\frac{\partial}{\partial w_2}, 
 $$
 namely, $g_0(\alpha, X)=0$ for all $X\in T^{1, 0}(\S^{3}\times \S^{1})$. Write $\alpha=\zeta_{1}+\ii \zeta_{2},$ where
$$
\zeta_{1}=B\frac{\partial }{\partial x_{2}}-\frac{x_{2}}{2\pi}\frac{\partial}{\partial x_{0}}+\frac{x_{5}}{2\pi}\frac{\partial }{\partial x_{3}},\,\,\,\,\,\,\,  \zeta_{2}=B\frac{\partial }{\partial x_{5}}+\frac{x_{5}}{2\pi}\frac{\partial}{\partial x_{0}}+\frac{x_{2}}{2\pi}\frac{\partial }{\partial x_{3}}
$$
where $w_0=x_0+\ii x_3,w_1=x_1+\ii x_4,w_2=x_2+\ii x_5$. Since $\tilde{g}_\epsilon$ is Hermitian,
\begin{equation}\label{estimate}
\tilde{g}_{\epsilon}(\zeta_{1},\zeta_{1})=\tilde{g}_{\epsilon}(\zeta_{2},\zeta_{2})=\frac{1}{2}\tilde{g}_{\epsilon}(\alpha,\alpha)= \frac{1}{2}\epsilon^{2}\,\,\,\text{ and}\,\,\,\,\tilde{g}_{\epsilon}(\zeta_{1},\zeta_{2})=0.
\end{equation}

	Now, we will find the integral curve $\gamma(t)$ of $a\zeta_{1}+b\zeta_{2}$ in $\S^3\times\S^3$, where $a, b$ are real constants satisfying $a^{2}+b^{2}=1$ to be determined later. In the local chart $U_{00}$, assume $\gamma(t)=(x_{0}(t), \cdots, x_{5}(t))$. From $d\gamma/dt = a\zeta_1+b\zeta_2$, we have
	\begin{equation}\label{integral curve}
	\small\left\{ \begin{aligned}
	&\frac{dx_{2}}{dt}=aB, &\frac{dx_{5}}{dt}=bB,\\
	&\frac{dx_{0}}{dt}=\frac{b}{2\pi}x_{5}-\frac{a}{2\pi}x_{2},\ &\frac{dx_{3}}{dt}=\frac{a}{2\pi}x_{5}+\frac{b}{2\pi}x_{2},\\
	&\frac{dx_{1}}{dt}=0,&\frac{dx_{4}}{dt}=0.
	\end{aligned}  \right.
	\end{equation}
	Let $x_{2}=r\cos\theta, x_{5}=r\sin\theta.$ Then
	\begin{equation*}\left\{
	\begin{aligned}
	&\frac{d x_{2}}{dt}=a(1+r^{2}),\\
	&\frac{d x_{5}}{dt}=b(1+r^{2}).
	\end{aligned} \right.
	\end{equation*}
	Define $\nu_{ab}$ such that $\cos \nu_{ab}=b, \sin\nu_{ab}=a.$
	It follows
	\begin{equation*}\left\{
	\begin{aligned}
	&\frac{dr}{dt}=\sin(\theta+\nu_{ab})(1+r^{2}),\\
	&r\frac{d\theta}{dt}=\cos(\theta+\nu_{ab})(1+r^{2}).
	\end{aligned}\right.
	\end{equation*}
	When $\cos(\theta(0)+\nu_{ab})=0$, it is easy to check
		\begin{equation*}\left\{
		\begin{aligned}
		&r=\sin(\theta+\nu_{ab}) \tan(t+\sin(\theta+\nu_{ab}) a_{1}),\\
		&\theta=\theta(0)
		\end{aligned}\right.
		\end{equation*}
		where $a_{1}=\arctan r(0)$ is a solution. As long as $x_2,x_5$ are found, the functions $x_0,x_1,x_3,x_4$ can be solved from \eqref{integral curve} uniquely for the initial data.
	
	For any given $(w_{0}(0), w_{1}(0), w_{2}(0))$, we may choose $a, b$ so that $\cos(\nu_{ab}+\theta(0))=0, \sin (\nu_{ab}+\theta(0))=-1$. Let $\gamma(t)$ be the integral curve solving \eqref{integral curve} with $\gamma(0)=(w_{0}(0), w_{1}(0), w_{2}(0))$. From the discussion above,
	$r= -\tan(t-a_{1})$ and it follows that $r(a_1)=0$, in turn, $\gamma(a_{1})\in \S^{3}\times \S^{1}$. By \eqref{estimate}
	$$
	\left|\frac{d\gamma}{dt}\right|_{\tilde{g}_\epsilon}\leq C\epsilon.
	$$
	 Hence the length of $\gamma(t)$ for $0\leq t\leq a_{1}$ is less than $C\epsilon$. Therefore, the Gromov-Hausdorff distance between $(\S^{3}\times \S^{3}, \tilde{g}_{\epsilon})$ and $(\S^{3}\times \S^{1}, {g}_{\epsilon})$ is less than $C\epsilon$.
		
	Next, we show that ${g}_{\epsilon}\rightarrow {g}_{0}$ as $\epsilon\rightarrow 0$ in the $C^{0}$ norm in the metric ${g}_1$ on $\S^3\times\S^1$.
	 In fact, by \eqref{induced metric 1}
	 $$
	 {g}_{\epsilon}-{g}_{0}=\epsilon^2
	 \begin{pmatrix}
	 \pi^{2} \   & \   \frac{\ii\pi}{2} \frac{w_{1}}{A}\\[3mm]
	 -\frac{\ii \pi}{2} \frac{\overline{w}_1}{A}\ &\ \frac{1}{4}\frac{|w_{1}|^2}{A^2}\\[3mm]
	 \end{pmatrix}.
	$$
	 Direct computation exhibits
	 $$
	 \begin{aligned}
	 &\left({g}_{\epsilon,0\overline j}-{g}_{0,0\overline j}\right){g}_{0}^{k\overline{j}}=\epsilon^{2}\delta_{0k},\\
	&\left({g}_{\epsilon,1\overline{j}}-{g}_{0,1\overline j}\right){g}_{0}^{0\overline{j}}=-\epsilon^{2}\frac{\ii}{2\pi}\frac{\overline{w}_1}{A},\\
	& \left({g}_{\epsilon,1\overline{j}}-{g}_{0,1\overline j}\right){g}_{0}^{1\overline{j}}=0,
	 \end{aligned}
	  $$
	  where
	  $$
	 {g}_{0}^{-1}=\frac{1}{\pi^{2}}
	 \begin{pmatrix}
	 \frac{A}{4}+\frac{3}{4}\ &\ -\frac{\ii\pi}{2}Aw_{1}\\[3mm]
	 \frac{\ii\pi}{2}A\overline{w}_{1}\ &\ \pi^{2}A^{2}\\[3mm]
	 \end{pmatrix}.
	 $$
	 Hence
\begin{equation*}
\begin{split}
	 &|{g}_{\epsilon}-{g}_0|^{2}_{ g_{0}}=g_{0}^{i\overline{j}}g_{0}^{k\overline{l}}\left({g}_{\epsilon,i\overline l}-{g}_{0,i\overline l}\right)\left({g}_{\epsilon,k\overline j}-{g}_{0,k\overline j}\right)\\[2mm]
	 =&g_{0}^{0\overline{j}}g_{0}^{0\overline{l}}\left({g}_{\epsilon,0\overline l}-{g}_{0,0\overline l}\right)\left({g}_{\epsilon,0\overline j}-{g}_{0,0\overline j}\right)= \epsilon^{4}\rightarrow 0, \,\,\text{as $\epsilon\rightarrow0$}.
	 \end{split}
\end{equation*}
	 So $(\S^{3}\times \S^{1}, g_{\epsilon})$ converge to $(\S^{3}\times \S^{1}, {g}_{0})$ in the Gromov-Hausdorff distance.

\vspace{.1cm}
	
	Finally, by the triangle inequality,
	\begin{equation*}
	\begin{aligned}
	d_{GH}&\left((\S^{3}\times \S^{3}, \tilde{g}_{\epsilon}), (\S^{3}\times \S^{1}, {g}_{0})\right)\\
	&\leq d_{GH}\left((\S^{3}\times \S^{3}, \tilde{g}_{\epsilon}), (\S^{3}\times \S^{1}, {g}_{\epsilon})\right)+d_{GH}\left((\S^{3}\times \S^{1}, {g}_{\epsilon}), (\S^{3}\times \S^{1}, {g}_{0})\right)\\
	&\leq C\epsilon
	\end{aligned}
	\end{equation*}
	where $d_{GH}$ denotes the Gromov-Hausdorff distance.
	We conclude that $(\S^{3}\times \S^{3}, g_{\epsilon})$ converge to $(\S^{3}\times \S^{1}, {g}_{0})$ in the Gromov-Hausdorff topology as $\epsilon\to 0^+$.
\end{proof}

\subsection{Cohomology classes of the fundamental forms $\omega_\epsilon$} 	
	
	The Aeppli cohomology on a compact complex manifold of complex dimension $n$ is defined by
	$$
	H_{A}^{p,q}=\frac{(Ker \ddbar)\cap \Omega^{p,q}}{(Im\ \partial+Im\  \overline\partial)\cap\Omega^{p,q}}
	$$
	where $\Omega^{p,q}$ is the space of $(p,q)$ forms. The Hodge $*$ operator associated to a Hermitian metric is isomorphic from the Bott-Chern cohomology $H^{n-p,n-q}_{BC}$ to $H^{p,q}_A$ (cf. \cite[Theorem 2.5]{An}). On a complex manifold with a Hermitian metric $g$, the Chern Ricci curvature (of first type) is 
	$$
	Ric({g}):=-\sqrt{-1}\partial\overline\partial\log\det g
	$$
	and the Chern scalar curvature is 
	$$
	R(g):=g^{i\bar j}Ric_{i\bar j}.
	$$
Both $g$ and the complex structure $J$ on $\S^3\times\S^1$ are parallel with respect to the Chern connection. More information can be found, for example, in \cite{LY}. 	 	

	\vspace{.1cm}
	
	A Hermitian metric on a complex $n$-dimensional manifold is called a \textit{Gauduchon metric} if its fundamental class $\omega$ satisfies $\ddbar\omega^{n-1}=0$. We have
\begin{prop}\label{gauduchon metric}
	Let ${\omega}_{\epsilon}=\ii {g}_{\epsilon, i\bar{j}}dz_{i}\wedge d{\bar{z}^{j}}$ be the K\"{a}hler form associated to ${g}_{\epsilon}$. Then
	
	(1) ${\omega}_{\epsilon}$ is a Gauduchon metric, i.e., $\partial\overline\partial{\omega}_{\epsilon}=0$, and represents an element in $H^{1,1}_A$.
	
	(2) $Ric({g}_{\epsilon})=2\ii A^{-2}dw_{1}\wedge d\overline{w}_{1}$ and $R({g}_{\epsilon})=2$.
	
	(3) $[{\omega}_{\epsilon}]\neq 0$ in $H^{1,1}_{A}$  and $[{\omega}_{1}]=[\frac{2}{1+\epsilon^{2}}{\omega}_{\epsilon}]$. In particular,
	$[{\omega}_{\epsilon_1}]\not=[{\omega}_{\epsilon_2}]$ if $\epsilon_1\not=\epsilon_2$.
	
\end{prop}

\begin{proof}
	(1) From \eqref{induced metric}, when  $(i_{0},j_{0})=(0,0) \text{\ or\ } (1,1)$, it is obvious that
	$$
	\partial\overline\partial \left({g}_{\epsilon, i_{0}\overline{j_{0}}}dw_{i_{0}}\wedge d\overline{w}_{j_{0}}\right)=0;
	$$

	when  $(i_{0},j_{0})=(0,1)$, we have
	\begin{equation*}
	\partial\overline\partial\left({g}_{\epsilon,0\overline{1}}dw_{0}\wedge d\overline{w}_{1}\right)=\partial_{1}\partial_{\overline{1}}\left(\frac{(1+\epsilon^{2})\ii\pi}{2} \frac{w_1}{A}dw_{0}\wedge d\overline{w}_{1}\right)=0
	\end{equation*}
	as ${g}_{\epsilon, 0\overline{1}}$ only depends on $w_{1}, \overline{w}_{1}$; similarly,
	\[
	\partial\overline\partial\left({g}_{\epsilon, 1\overline{0}}d w_1\wedge d\overline{w}_{0}\right)=0.
	\]

(2) From \eqref{induced metric 1} direct computation leads to
$$
\begin{aligned}
&\det({g}_{\epsilon})=\left(1+\epsilon^{2}\right)\pi^{2}A^{-2}, \\
&Ric({g}_{\epsilon})=-\sqrt{-1}\partial\overline{\partial}\log\det ({g}_\epsilon)=2\sqrt{-1} A^{-2}dw_{1}\wedge d\overline{w}_{1}.\\
\end{aligned}
$$
By \eqref{induced metric 1}, the inverse matrix is
\begin{equation}\label{inverse metric 1}
\big({g}_{\epsilon}^{i\bar{j}}\big)=\frac{A^{2}}{(1+\epsilon^{2})\pi^{2}}\left(
\begin{matrix}
\frac{1}{4A}+\frac{3}{4A^2}+\frac{\epsilon^{2}}{4}\frac{|w_{1}|^2}{A^2} & \ \ -\frac{(1+\epsilon^{2})\ii\pi}{2} \frac{w_1}{A} \\[3mm]
\frac{(1+\epsilon^{2})\ii\pi}{2}\frac{\overline{w}_{1}}{A}& \ \ (1+\epsilon^{2})\pi^2 \\[3mm]
\end{matrix}\right).
\end{equation}

Hence, the Chern scalar curvature is constant for the family of metrics $g_\epsilon$:
$$
R({g}_\epsilon)=g_{\epsilon}^{i\bar j}Ric({g}_\epsilon)_{i\bar j}=2.
$$

(3) The global (1,1) form
$$
{\omega}_{1}-\frac{2}{1+\epsilon^{2}}{\omega}_{\epsilon}=\frac{\epsilon^{2}-1}{1+\epsilon^{2}}\ii A^{-2}dw_{1}\wedge \overline{w}_{1}=\frac{\epsilon^{2}-1}{2(1+\epsilon^{2})}Ric({\omega}_{1})
$$
is $d$-closed; hence, it is $d$-exact since $d=\partial+\overline\partial$ and the cohomology group $H^{2}(\S^{3}\times \S^{1})=0$. So
$
[{\omega}_{1}]=[\frac{2}{1+\epsilon^{2}}{\omega}_{\epsilon}]
$ in $H^{1,1}_A$.

By \cite[Proposition 37]{HL}, $[{\omega}_{1}]\neq 0$ in $H^{1,1}_{A}$. In fact, if the real $(1,1)$ form $\omega_\epsilon$ were in $Im\,\partial+Im\,\overline\partial$ then it could be expressed as $\partial\alpha+\overline{\partial}\beta$ for some $\alpha\in\Omega^{0,1},\beta\in\Omega^{1,0}$. Then taking $S=\alpha+\beta$ and applying Stokes' theorem
 $$
 \begin{aligned}
 0=&\int_{\S^3\times\S^1} d(S\wedge \overline{d S})=\int_{\S^3\times\S^1} dS\wedge \overline{d S}\\=&\int_{\S^3\times\S^1}\tilde\omega_\epsilon\wedge\tilde\omega_\epsilon+\int_{\S^3\times\S^1}dS^{0,2}\wedge \overline{dS^{0,2}}+\int_{\S^3\times\S^1}dS^{2,0}\wedge \overline{dS^{2,0}}.
 \end{aligned}
 $$
 Since all three integrals are non-negative, and $\omega_\epsilon\wedge\omega_\epsilon\geq 0$, we must have
 $\tilde\omega_\epsilon\wedge\tilde\omega_\epsilon\equiv 0$, but this contradicts that $\omega_\epsilon$ is the fundamental class. It follows that $[\omega_{\epsilon_2}]-[\omega_{\epsilon_1}]=[\frac{\epsilon^2_2-\epsilon^2_1}{2}\omega_1]\not=0$ in $H^{1,1}_A$, whenever $\epsilon_2\not=\epsilon_1$.
\end{proof}

We say a Hermitian metric $g$ is a \textit{locally conformally K\"{a}hler} metric on a complex manifold $M$ if for each point $x\in M$ there exists an open neighbourhood $U$ of $x$ and a function $f$ on $U$ so that $e^{-f}g$ is K\"{a}hler. It is equivalent to that the Lee form, which is a real 1-form defined by
\begin{equation}\label{def lee form}
d\omega = -2\theta \wedge \omega,
\end{equation}
is closed \cite{GO}.
We have
\begin{theo}
	Each metric $g_{\epsilon}$ is locally conformally K\"{a}hler.
\end{theo}

\begin{proof}
	By \cite{GO}, it suffices to prove that the Lee form $\theta$ is closed. Let $\omega_{\epsilon}$ be the K\"{a}hler form corresponding to $g_{\epsilon}$ and $\theta_\epsilon$ be its Lee form. Then
	\begin{equation*}
	\begin{split}
	\omega_{\epsilon}=&
	(1+\epsilon^{2})\pi^2 \i dw_{0}\wedge d\overline{w}_{0} +\frac{(1+\epsilon^{2})\i\pi}{2} \frac{w_1}{A}\i dw_{0}\wedge d\overline{w}_{1} \\
	&- \frac{(1+\epsilon^{2})\i\pi}{2}\frac{\overline{w}_{1}}{A}\i dw_{1}\wedge d\overline{w}_{0}
	+ \left(\frac{1}{4A}+\frac{3}{4A^2}+\frac{\epsilon^{2}}{4}\frac{|w_{1}|^2}{A^2}\right)\i dw_{1}\wedge d\overline{w}_{1}.
	\end{split}
	\end{equation*}
	Hence, we have
	\begin{equation}\label{domega}
	\begin{split}
	d\omega_{\epsilon}&=\i\frac{(1+\epsilon^{2})\i\pi}{2}\frac{A-|w_{1}^{2}|}{A^{2}}\left(d w_{1}\wedge dw_{0}\wedge d\overline{w}_{1}-d\overline{w}_{1}\wedge d w_{1}\wedge  d\overline{w}_{0}\right)\\
	&=-\frac{(1+\epsilon^{2})\i\pi}{2}\frac{1}{A^{2}}\left(\i dw_{0}\wedge d w_{1}\wedge d\overline{w}_{1}-\i d\overline{w}_{0}\wedge d w_{1}\wedge d\overline{w}_{1}\right).\\
	\end{split}
	\end{equation}
	Assume $$\theta_{\epsilon}=a\,d w_{0}+\overline{a}\,d\overline{w}_{0}+b\,dw_{1}+\overline{b}\,d\overline{w}_{1}.$$ 
	It follows
	\begin{equation*}
	\begin{split}
	\theta_{\epsilon}\wedge\omega_{\epsilon}=&\left(a \frac{(1+\epsilon^{2})\ii\pi}{2}\frac{\overline{w}_{1}}{A}+b(1+\epsilon^{2})\pi^2\right)\i dw_{0}\wedge d\overline{w}_{0}\wedge d w_{1}\\
	&+\left(-\overline{a}\frac{(1+\epsilon^{2})\i\pi}{2} \frac{w_1}{A}+\overline{b}(1+\epsilon^{2})\pi^2 \right) \i dw_{0}\wedge d\overline{w}_{0}\wedge  d\overline{w}_{1}\\
	&+\left(a\frac{A+3+\epsilon^{2}|w_{1}|^2}{4A^2}-b\frac{(1+\epsilon^{2})\ii\pi}{2} \frac{w_1}{A}\right)\i d w_{0}\wedge d w_{1}\wedge d\overline{w}_{1}\\
	&+\left(\overline{a}\frac{A+3+\epsilon^{2}|w_{1}|^2}{4A^2}+\overline{b}\frac{(1+\epsilon^{2})\ii\pi}{2} \frac{\overline{w}_1}{A}\right)\i d\overline{w}_{0}\wedge d w_{1}\wedge d\overline{w}_{1}.
	\end{split}
	\end{equation*}
	By \eqref{def lee form} and \eqref{domega}, we demand 
	\begin{equation*}
	\begin{split}
	&a \frac{(1+\epsilon^{2})\i\pi}{2}\frac{\overline{w}_{1}}{A}+b(1+\epsilon^{2})\,\pi^2=0,\\
	&a\left(\frac{A+3+\epsilon^{2}|w_{1}|^2}{4A^2}\right)-b\frac{(1+\epsilon^{2})\ii\pi}{2} \frac{w_1}{A}=\frac{(1+\epsilon^{2})\i\pi}{4}\frac{1}{A^{2}}.
	\end{split}
	\end{equation*}
	Then
	\begin{equation*}
	a=(1+\epsilon^{2})\frac{\i\pi}{4}, \,\,\, \, b=(1+\epsilon^{2})\frac{\overline{w}}{8A}.
	\end{equation*}
	We obtain
	\begin{equation}\label{lee form}
	\theta_{\epsilon}=(1+\epsilon^{2})\left(\frac{\i\pi}{4}dw_{0}-\frac{\i\pi}{4}d\overline{w}_{0}+\frac{\overline{w}_{1}}{8A}dw_{1}+\frac{w_{1}}{8A}d\overline{w}_{1}\right).
	\end{equation}
	Therefore
	\begin{equation*}
	d\theta_{\epsilon}=\frac{1}{8}(1+\epsilon^{2})\left(\left(\frac{1}{A}-\frac{|w_{1}|^{2}}{A^{2}}\right)d\overline{w}_{1}\wedge dw_{1}+\left(\frac{1}{A}-\frac{|w_{1}|^{2}}{A^{2}}\right)dw_{1}\wedge d\overline{w}_{1}\right)=0.
	\end{equation*}
	Thus $g_\epsilon$ is locally conformally K\"ahler. 
\end{proof}

Now, we demonstrate that the metrics $g_\epsilon$ are different from those in \cite{GO}.

\vspace{.1cm}

First we recall the construction of the Hermitian metrics in \cite{GO} and adopt the notations therein. The Hopf surface $M_{\alpha_{1},\alpha_{2}}$ is $(\mathbb{C}^{2}\setminus\{(0,0)\})/\sim$, where $$(u,v)\sim(\alpha_{1}u+\lambda v^{m },\alpha_{2}v), \,\,\,(u,v)\in \mathbb{C}^{2}\setminus\{(0,0)\} $$ and
$\alpha_{1},\alpha_{2}, \lambda$ are complex number and $m$ is a nonnegative integer such that
$$|\alpha_{1}|\geq |\alpha_{2}|>1$$
and
$$(\alpha_{1}-\alpha_{2}^{m})\lambda=0.$$
Let $\Phi_{\alpha_{1},\alpha_{2}}$ be a function that satisfies
\begin{equation*}
|u|^{2}\Phi_{\alpha_{1},\alpha_{2}}^{-\frac{2k_{1}}{k_{1}+k_{2}}}+|v|^{2}\Phi_{\alpha_{1},\alpha_{2}}^{-\frac{2k_{2}}{k_{1}+k_{2}}}=1,
\end{equation*}
where $k_{i}=\ln |\alpha_{i}|, i=1,2.$
It is shown that $\frac{1}{\Phi_{\alpha_{1},\alpha_{2}}}\ddbar \Phi_{\alpha_{1},\alpha_{2}}$ are well defined and  locally conformally K\"ahler metrics on $M_{\alpha_1,\alpha_2}$ with parallel Lee form.

\vspace{.1cm}

\begin{prop}\label{parallel}
	The Lee form $\theta_{\epsilon}$ is not parallel.
\end{prop}

\begin{proof} Assume $w_0=x_0+\ii x_2,w_1=x_1+\ii x_3$.
	By \eqref{lee form}, we have
	\[\theta_{\epsilon}=(1+\epsilon^{2})\left(-\frac{\pi}{2}dx_{2}+\frac{1}{4A}(x_{1}dx_{1}+x_{3}dx_{3})\right).\]
	By \eqref{christoeffl},  $g_{ij}$ are independent of $x_{0}, x_{2}$ and $g_{ij}, i, j=0,2$ are constant. Then 
	\begin{equation*}
	\begin{split}
	\nabla_{\frac{\partial}{\partial x_{0}}}\theta_{\epsilon}=&-(1+\epsilon^{2})\sum_{p=0}^{3}\left(\frac{x_{1}}{4A}\Gamma_{p0}^{1}+\frac{x_{3}}{4A}\Gamma^{3}_{p0}\right)dx_{p}\\
	=&-(1+\epsilon^{2})\frac{1}{8A}\sum_{p, q=0}^{3}\left((x_{1}g^{1q}(g_{0q,p}-g_{0p,q}))+x_{3}g^{3q}(g_{0q,p}-g_{0p,q})\right)dx_{p}\\
	=&-\sum_{p,q=0}^3(1+\epsilon^{2})\frac{1}{8A}(x_{1}g^{1q}+x_{3}g^{3q})(g_{0q,p}-g_{0p,q})dx_{p}\\
	=&-\sum_{q=1,3}(1+\epsilon^{2})\frac{1}{8A}(x_{1}g^{1q}+x_{3}g^{3q})(g_{0q,1}-g_{01,q})dx_{1}\\
	&-\sum_{q=1,3}(1+\epsilon^{2})\frac{1}{8A}(x_{1}g^{1q}+x_{3}g^{3q})(g_{0q,3}-g_{03,q})dx_{3}.
	\end{split}
	\end{equation*}	
	By \eqref{derivative of metric} and $g^{13}=0$ given by \eqref{real inverse metric}, we have
	\begin{equation*}
	\begin{split}
	-\sum_{q=1,3}(x_{1}g^{1q}+x_{3}g^{3q})(g_{0q,1}-g_{01,q})
	=&(x_{1}g^{13}+x_{3}g^{33})(g_{01,3}-g_{03,1})\\
	=&x_{3}g^{33}\frac{1+\epsilon^{2}}{4}\frac{2\pi}{A^{2}}(x_{3}^{2}-x_{1}^{2}).
	\end{split}
	\end{equation*}
	Then $\nabla_{\frac{\partial}{\partial x_{0}}}\theta_{\epsilon}\not\equiv 0$, as required.
\end{proof}

\begin{rem}
	First, by \cite[Theorem 1]{GO}, the Lee form of  $\omega_{\alpha_1,\alpha_2}=\frac{1}{\Phi_{\alpha_{1},\alpha_{2}}}\ddbar \Phi_{\alpha_{1},\alpha_{2}}$ is parallel. Therefore, there is no diffeomorphism pulling back the metric $\omega_{\alpha_1,\alpha_2}$ to $\omega_{\epsilon}$.
	Second, if
	$\omega_{\alpha_1,\alpha_2}$ is a Gauduchon metric, then its conformal class must be different from that of our metric $\omega_\epsilon$. In fact, by Proposition \ref{gauduchon metric}, $\omega_{\epsilon}$ is Gauduchon.
However, the Gauduchon metric is unique up to a constant scaling in a conformal class (by \cite{GP} or \cite[Theorem 1.2.4]{LT}). 	
\end{rem}

\section{Stability of the toric fibres as harmonic map and minimal surface} Continue to
 let $\pi:\S^3\times\S^1\to\C P^1$ be the Hopf fibration. For any point $p$ in the base $\C P^1$, denote $\T^2_p=\pi^{-1}(p)$ and $p_0,p_1$ for the two special tori defined in the introduction. As the fibre $\T^2_p$ is a complex submanifold, let $c_p$ be the conformal structure on it such that the inclusion mapping $f_p:\T^2_p\hookrightarrow\S^3\times\S^1$ is holomorphic with respect to $c_p$. It is well known that harmonicity of a mapping from a Riemann surface $\Sigma$ only depends on the conformal class $[c]$ of $\Sigma$, not on specific metrics within $[c]$. This leads to further simplification of \eqref{harmonic map}, when the domain is a Riemann surface, since we can use the isothermal coordinates associated with $[c]$. 
In fact, the complex structure of $\T^2_p$ is the standard one on the unit square torus $\T^2$ by Proposition \ref{complex coordinates} and the induced metric is 
the standard metric on $\T^2$ scaled by $(1+\epsilon^2)\pi^2$ by \eqref{induced metric 1}. 

	
\subsection{Curvature of $(\S^3\times\S^1, {g}_\epsilon)$}
	We have constructed a family of Hermitian metrics $g_\epsilon$ for $\epsilon\geq 0$ on $\S^3\times\S^1$. We will be interested in the isotropic curvatures of the isotropic 2-planes spanned by $T^{1,0}$-vectors. The observation is that each of these curvatures is nonnegative and vanish exactly on the two special tori $\T^2_{p_0},\T^2_{p_1}$. The properties of the Riemannian curvature tensor contained in Lemma \ref{curvature} are crucial for our stability analysis. 
	
	Let us first consider a local picture. In $V_{00}$, the inclusion map $\S^3\times\S^1\hookrightarrow\C^2\times\C^1$ can be written as
\begin{equation}\label{2d map}
\left\{\begin{aligned}
z_{0}\ &=\ A^{-\frac{1}{2}}e^{\ii\pi(w_{0}+\overline{w}_{0})},\\
z_{1}\ &=\ z_{0}w_{1}=w_{1}A^{-\frac{1}{2}}e^{\ii\pi(w_{0}+\overline{w}_{0})},\\
z'_{0}\ &=\ e^{\pi(w_{0}-\overline{w}_{0})-\frac{1}{2}\ii\log A}
\end{aligned}\right.
\end{equation}	
with $A=1+|w_1|^2$. When $w_1=0$, we get $(z_0,z_1,z'_0)=(e^{\ii\pi(w_0+\overline{w}_0)},0,e^{\ii(w_0-\overline{w_0})})$ tracing a square torus $\S^1\times\{0\}\times\S^1\subset\C\times\C\times\C$. Similarly, in $V_{10}$ there is another square torus $\{0\}\times\S^1\times\S^1\subset\C\times\C\times\C$.
In fact, these are the only tori in the holomorphic toric fibration $\pi$ as a direct product of a great circle in $\S^3$ with $\S^1$. It turns out that the isotropic curvature vanishes on and only on these two square tori in $\S^3\times\S^1\subset\C^2\times\C$.

Let $w_0=x_0+\ii x_2,w_1=x_1+\ii x_3$. Set $R_{ijkl}=R\left(\frac{\partial }{\partial x_{i}},\frac{\partial }{\partial x_{j}},\frac{\partial }{\partial x_{k}},\frac{\partial }{\partial x_{l}}\right)$, where
\begin{equation*}
\begin{split}
R_{ijkl}&=-g\left(\nabla_{\frac{\partial}{\partial x_{i}}}\nabla_{\frac{\partial}{\partial x_{j}}}\frac{\partial}{\partial x_{k}}-\nabla_{\frac{\partial}{\partial x_{j}}}\nabla_{\frac{\partial}{\partial x_{i}}}\frac{\partial}{\partial x_{k}}
-\nabla_{[\frac{\partial}{\partial x_{i}}, \frac{\partial}{\partial x_{j}}]}\frac{\partial}{\partial x_{k}}, \frac{\partial}{\partial x_{l}}\right)\\
&=-\frac{1}{2}\left(g_{ik,jl}+g_{jl,ik}-g_{jk,il}-g_{il,jk}\right)-\Gamma^{p}_{ik}\Gamma^{q}_{jl}g_{pq}+\Gamma^{p}_{jk}\Gamma^{q}_{il}g_{pq}.
\end{split}
\end{equation*}

First we prove
\begin{lem}\label{curvature}
\begin{equation}\label{section curvature}
\begin{aligned}
R_{0101}&=\Gamma^{p}_{10}\Gamma^{q}_{10}g_{pq}=\frac{1}{4}g^{33}(g_{03,1}-g_{01,3})^{2},\\
R_{0303}&=\Gamma^{p}_{30}\Gamma^{q}_{30}g_{pq}=\frac{1}{4}g^{11}\left(g_{30,1}-g_{10,3}\right)^{2},\\
R_{2121}&=\Gamma^{p}_{12}\Gamma^{q}_{12}g_{pq}=\frac{1}{4}g^{33}\left(g_{21,3}-g_{32,1}\right)^{2},\\
R_{2323}&=\Gamma^{p}_{32}\Gamma^{q}_{32}g_{pq}=\frac{1}{4}g^{11}\left(g_{21,3}-g_{32,1}\right)^{2},\\
R_{0103}&=\Gamma^{p}_{01}\Gamma^{q}_{03}g_{pq}=0,\\
R_{2123}&=\Gamma^{p}_{21}\Gamma^{q}_{23}g_{pq}=0\\
\end{aligned}
\end{equation}
and
\begin{equation*}
R_{0213}=0, \ \ R_{0321}=0, \ \ R_{0123}=0.
\end{equation*}
\end{lem}		
\begin{proof}
	The underlying Riemannian metric of the Hermitian metric $g_\epsilon$ is given locally in $U_{00}$ by the $4\times 4$ real matrix (now $0\leq i,j\leq 3$)
	\begin{equation}\label{metric real}
	\left(g_{ij}\right)=\frac{1}{2}
	\begin{pmatrix}
	&(1+\epsilon^{2})\pi^2\ &\  -\frac{1+\epsilon^{2}}{2}  \frac{\pi x_{3}}{A}&0 &-\frac{1+\epsilon^{2}}{2} \frac{\pi x_{1}}{A} \\[3mm]
	&-\frac{1+\epsilon^{2}}{2} \frac{\pi x_{3}}{A} \ \ \ &  \ \ \  \frac{1}{4A}+\frac{3+\epsilon^{2}|w_1|^2}{4A^2}\ & \frac{1+\epsilon^{2}}{2}\frac{\pi x_{1}}{A} &\ 0\\[3mm]
	&0 &\frac{1+\epsilon^{2}}{2} \frac{\pi x_1}{A}\ &\ \ (1+\epsilon^{2}) \pi^2\ &\  -\frac{1+\epsilon^{2}}{2}  \frac{\pi x_3}{A} \\[3mm]
	&-\frac{1+\epsilon^{2}}{2} \frac{\pi x_{1}}{A}\ &0&\ \   -\frac{1+\epsilon^{2}}{2}  \frac{\pi x_3}{A}\ \ \ &\ \ \ \ \ \frac{1}{4A}+\frac{3+\epsilon^{2}|w_1|^2}{4A^2}\\[3mm]
	\end{pmatrix}.
	\end{equation}
	We also have
		\begin{equation}\label{real inverse metric}
	\left(g_{ij}\right)^{-1}=\frac{2A^{2}}{(1+\epsilon^{2})\pi^{2}}
	\begin{pmatrix}
	&\frac{1}{4A}+\frac{3+\epsilon^{2}|w_1|^2}{4A^2}\ &\  \frac{1+\epsilon^{2}}{2}  \frac{\pi x_{3}}{A}&0 &\frac{1+\epsilon^{2}}{2} \frac{\pi x_{1}}{A} \\[3mm]
	&\frac{1+\epsilon^{2}}{2}  \frac{\pi x_{3}}{A} \ \ \ &  \ \ \ (1+\epsilon^{2})\pi^2\   & -\frac{1+\epsilon^{2}}{2}\frac{\pi x_{1}}{A} &\ 0\\[3mm]
	&0 &-\frac{1+\epsilon^{2}}{2} \frac{\pi x_1}{A}\ &\ \ \ \ \frac{1}{4A}+\frac{3+\epsilon^{2}|w_1|^2}{4A^2}\ &\  \frac{1+\epsilon^{2}}{2}  \frac{\pi x_3}{A} \\[3mm]
	&\frac{1+\epsilon^{2}}{2} \frac{\pi x_{1}}{A}\ &0&\ \   \frac{1+\epsilon^{2}}{2} \frac{\pi x_3}{A}\ \ \ &\ \ \ \ \ (1+\epsilon^{2})\pi^2\\[3mm]
	\end{pmatrix}.
	\end{equation}
	
	Now, we compute the five curvature terms in \eqref{isotropic}. First, we calculate $R_{0101}$. Evidently the entries $g_{ij}$ are independent of $x_{0}$ and $x_{2}$. Therefore
	\begin{equation}\label{christoeffl}
	\begin{aligned}
	&\Gamma_{00}^{i}=\frac{1}{2}g^{il}(g_{0l, 0}+g_{0l, 0}-g_{00,l})=-\frac{1}{2}g^{il}g_{00,l}=0,\\
	&\Gamma^i_{02} =\frac{1}{2}g^{il}\left(g_{0l,2}+g_{2l,0}-g_{02,l}\right)=-\frac{1}{2} g^{il}g_{02,l}=0,\\
	&\Gamma^i_{22}=\frac{1}{2}g^{il}\left(2g_{2l,2}-g_{22,l}\right)=-\frac{1}{2}g^{il}g_{22,l}=0.
	\end{aligned}
	\end{equation}
	Then, as $g_{pq}$ are constant for $p,q=0,2$ and $g_{ij,0},g_{ij,2}=0$, we have
	\begin{equation*}
	\begin{split}
	R_{0101}&=-\frac{1}{2}\left(g_{00,11}+g_{11,00}-g_{10,10}-g_{10,01}\right)-\Gamma^{p}_{00}\Gamma^{q}_{11}g_{pq}+\Gamma^{p}_{10}\Gamma^{q}_{10}g_{pq}\\
	&=\Gamma^{p}_{10}\Gamma^{q}_{10}g_{pq}\\
	&=\frac{1}{4}\sum_{p,q=0,1,2,3}g^{pq}(g_{1p,0}+g_{0p,1}-g_{01,p})(g_{1q,0}+g_{0q,1}-g_{01,q})\\
	&=\frac{1}{4}\sum_{p,q=1,3}g^{pq}(g_{0p,1}-g_{01,p})(g_{0q,1}-g_{01,q})\\
	&=\frac{1}{4}g^{33}(g_{03,1}-g_{01,3})^{2}.\\
	\end{split}
	\end{equation*}
	By the symmetry of the indices, we have
	\begin{equation}\label{curvature computation}
	\begin{aligned}
	R_{0303}&=\Gamma^{p}_{30}\Gamma^{q}_{30}g_{pq}=\frac{1}{4}g^{11}\left(g_{30,1}-g_{10,3}\right)^{2},\\
	R_{2121}&=\Gamma^{p}_{12}\Gamma^{q}_{12}g_{pq}=\frac{1}{4}g^{33}\left(g_{21,3}-g_{32,1}\right)^{2},\\
	R_{2323}&=\Gamma^{p}_{32}\Gamma^{q}_{32}g_{pq}=\frac{1}{4}g^{11}\left(g_{21,3}-g_{32,1}\right)^{2}.\\
	\end{aligned}
	\end{equation}
	In addition, by $g^{13}=0$, we get
	\begin{eqnarray*}
		\ &R_{0213}&=-\Gamma^{p}_{01}\Gamma^{q}_{23}g_{pq}+\Gamma^{p}_{12}\Gamma^{q}_{03}g_{pq} \nonumber\\
		&&=-\frac{1}{4}g^{ps}(g_{0s,1}+g_{1s,0}-g_{01,s})g^{qt}(g_{2t,3}+g_{3t,2}-g_{23,t})g_{pq}\nonumber\\
		&&\ \ \ \ +\frac{1}{4}g^{ps}(g_{1s,2}+g_{2s,1}-g_{12,s})g^{qt}(g_{0t,3}+g_{3t,0}-g_{03,t})g_{pq}\label{1 derivative}\\
		&&= -\frac{1}{4}g^{ps}(g_{0s,1}-g_{01,s})(g_{2p,3}-g_{23,p})+\frac{1}{4}g^{ps}(g_{2s,1}-g_{12,s})(g_{0p,3}-g_{03,p})\nonumber\\
		&&=-\frac{1}{4}g^{13}(g_{03,1}-g_{01,3})(g_{21,3}-g_{23,1})+\frac{1}{4}\sum_{s=1,3}g^{1s}(g_{2s,1}-g_{12,s})(g_{01,3}-g_{03,1})\nonumber\\
		&&=\frac{1}{4}g^{11}(g_{21,1}-g_{12,1})(g_{01,3}-g_{03,1})\nonumber=0. \nonumber
	\end{eqnarray*}
and 
	\begin{equation*}
\begin{split}
R_{0103}&=\Gamma^{p}_{10}\Gamma^{q}_{30}g_{pq}\\
&=\frac{1}{4}\sum_{p,q=0,1,2,3}g^{pq}(g_{1p,0}+g_{0p,1}-g_{01,p})(g_{3q,0}+g_{0q,3}-g_{03,q})\\
&=\frac{1}{4}\sum_{p,q=1,3}g^{pq}(g_{0p,1}-g_{01,p})(g_{0q,3}-g_{03,q})=-\frac{1}{4}g^{13}(g_{03,1}-g_{01,3})^{2}=0.\\
\end{split}
\end{equation*}
Similarly, we obtain $ R_{2123}=0,$ $R_{0321}=g^{13}(g_{03,1}-g_{01,3})(g_{21,3}-g_{31,1})=0,$ $R_{0123}=g^{13}(g_{01,3}-g_{03,1})(g_{23,1}-g_{21,1})=0.$
\end{proof}		
We also have:
	\begin{prop}\label{isotropy curvature}
		For any linearly independent $X, Y\in T^{1, 0}_p(\S^{3}\times \S^{1})$, $R\left(X, Y,\overline{X},\overline{Y}\right)\geq 0$ where $R$ is the Riemannian curvature operator and $p\in \S^{3}\times \S^{1}$, it equals to zero at $p\in \S^{3}\times \S^{1}$ if and only if $p\in\left(\S^{3}\cap (\{z_{0}=0\}\cup\{z_{1}=0\})\right)\times \S^{1} $.

\end{prop}

\begin{proof} Here and below we omit $\epsilon$ for simplicity of expressions.
	We first confine to the complex coordinate chart $(U_{00}, (w_0,w_1))$.

 Assume $X=a\frac{\partial}{\partial w_{0}}+b\frac{\partial}{\partial w_{1}}$ and $Y=
c\frac{\partial}{\partial w_{0}}+d\frac{\partial}{\partial w_{1}}$, for some complex numbers $a,b,c,d$ with $ab-cd\not=0$.
Then
$$
R\left(X, Y,\overline{X}, \overline{Y}\right)=|ad-bc|^{2}R\left(\frac{\partial}{\partial w_{0}}, \frac{\partial}{\partial w_{1}},{\frac{\partial}{\partial \overline w_{0}}},{\frac{\partial}{\partial \overline w_{1}}}\right).
$$

It suffices to prove that $R\left(\frac{\partial}{\partial w_{0}}, \frac{\partial}{\partial w_{1}},\frac{\partial}{\partial \overline{w}_{0}},\frac{\partial}{\partial \overline{w}_{1}}\right)\geq 0$.

Note that
\begin{equation}\label{isotropic}
\begin{split}
R\left(\frac{\partial}{\partial w_{0}}, \frac{\partial}{\partial w_{1}},\frac{\partial}{\partial \overline{w}_{0}},\frac{\partial}{\partial \overline{w}_{1}}\right)=&R\left(\frac{\partial }{\partial x_{0}},\frac{\partial }{\partial x_{1}},\frac{\partial }{\partial x_{0}},\frac{\partial }{\partial x_{1}}\right)+R\left(\frac{\partial }{\partial x_{0}},\frac{\partial }{\partial x_{3}},\frac{\partial }{\partial x_{0}},\frac{\partial }{\partial x_{3}}\right)\\
+R\left(\frac{\partial }{\partial x_{2}},\frac{\partial }{\partial x_{1}},\frac{\partial }{\partial x_{2}},\frac{\partial }{\partial x_{1}}\right)
&+R\left(\frac{\partial }{\partial x_{2}},\frac{\partial }{\partial x_{3}},\frac{\partial }{\partial x_{2}},\frac{\partial }{\partial x_{3}}\right)-2R\left(\frac{\partial }{\partial x_{0}},\frac{\partial }{\partial x_{2}},\frac{\partial }{\partial x_{1}},\frac{\partial }{\partial x_{3}}\right)
\end{split}
\end{equation}
where $w_0=x_0+\ii x_2,w_1=x_1+\ii x_3$.

To locate the zero locus of the curvature, we observe
\begin{equation}\label{derivative of metric}
\left\{\begin{aligned}
&g_{30,1}=\frac{1+\epsilon^{2}}{4} \left(-\frac{\pi}{A}+\frac{2\pi}{ A^2}x_{1}^{2}\right),\\
&g_{10,3}=\frac{1+\epsilon^{2}}{4}\left(-\frac{\pi}{A} +\frac{2\pi}{ A^2}x_{3}^{2}\right),\\
&g_{21,3}=-g_{32,1}=-\frac{1+\epsilon^{2}}{2}\frac{\pi}{ A^2}x_{1}x_{3}.
\end{aligned}\right.
\end{equation}
Hence, by Lemma \ref{curvature},
\begin{equation}
\begin{split}
R\left(\frac{\partial}{\partial w_{0}}, \frac{\partial}{\partial w_{1}},\frac{\partial}{\partial \overline{w}_{0}},\frac{\partial}{\partial \overline{w}_{1}}\right)&=\frac{1}{2}g^{11}\left((g_{03,1}-g_{01,3})^{2}+(g_{21,3}-g_{32,1})^{2}\right)\\
&=\frac{1}{2}g^{11}\frac{\pi^{2}(1+\epsilon^{2})^{2}}{4A^{4}}\left(x_{1}^{2}+x^{2}_{3}\right)^{2}\geq 0.
\end{split}
\end{equation}
It follows $R\left(\frac{\partial}{\partial w_{0}}, \frac{\partial}{\partial w_{1}},\frac{\partial}{\partial \overline{w}_{0}},\frac{\partial}{\partial \overline{w}_{1}}\right)=0$
if and only if
\begin{equation*}
x_{1}^{2}+x_{3}^{2}=0
\end{equation*}
that is $w_1=0$ hence $z_{1}=0$. Similarly, we have $z_0=0$ on $V_{10}$.

Therefore, the curvature vanishes on and only on $\S^3\cap\{z_1=0\} \times\S^1$ and $\S^3\cap\{z_0=0\}\times\S^1$. Both are embedded complex tori in $\S^3\times\S^1$.
\end{proof}

Theorem \ref{main1} follows from results established in this section.

\subsection{Stability of harmonic maps from $\T^2$ to the Hermitian Hopf surfaces}

Let $f:M\to N$ be a smooth mapping from a K\"ahler manifold $M$ with metric tensor $h_{\alpha\overline\beta}dz^\alpha d\overline{z}^\beta$ to a Riemannian manifold $N$ with metric tensor $g_{ij}dy^idy^j$. Using the complex coordinates on $M$, Sampson wrote the harmonic map equation \cite[p.129]{Sampson} as
\begin{equation}\label{harmonic map}
h^{\alpha\overline\beta}\left(\frac{\partial^2 f^i}{\partial z^\alpha\partial\overline{z}^\beta}+\Gamma^i_{jk}\frac{\partial f^j}{\partial z^\alpha}\frac{\partial f^k}{\partial z^{\overline\beta}}\right)=0
\end{equation}
where $\Gamma^i_{jk}$ are the Christoffel symbols of the Levi-Civita connection of $(N,g)$. This is convenient for verification of harmonicity of the holomorphic maps from tori in Theorem \ref{main3}.

\begin{theo}\label{harmonic torus}
 Each $f_p$ is a harmonic map from $(\T^2_p,c_p)$ to $(\S^3\times\S^1, {g}_\epsilon)$. There exist neighbourhoods of $p_0,p_1$ respectively such that $f_p$ is a stable harmonic map when $p$ belongs to the neighbourhoods.
 \end{theo}
\begin{proof}
First, we show that $f_p$ is harmonic for any fixed $p$. It suffices to verify the harmonic map equation locally. In the complex coordinate chart $(U_{00};w_0,w_1)$, the inclusion map may be written as $f_p(w_0)=(w_0,w_1(p))$. Now $w_0$ is a complex coordinate on $\T^2_p$ for the conformal structure $c_p$, the harmonic map equation \eqref{harmonic map} is
\begin{equation*}
\frac{\partial^2 f^i_p}{\partial w_0\partial\overline{w}_0}+\Gamma^i_{jk}\frac{\partial f^j_p}{\partial w_0}\frac{\partial f^k_p}{\partial \overline{w}_0}=0
\end{equation*}
where the indices $0\leq i,j,k\leq 3$ arise from $w_0=x_0+\ii x_2, w_1=x_1+\ii x_3$. Since $f_p$ is holomorphic, the functions $f^i_p$ are harmonic for any metric in the conformal class of $c_p$, hence the first term above vanishes. To see the second term also vanishes, since $f_p^1,f^3_p$ are constant it suffices to show $\Gamma^i_{00}, \Gamma^i_{02}, \Gamma^i_{22} = 0$.
Therefore, by \eqref{christoeffl}, $f_p$ is harmonic.	
	
Next, we examine the stability of these harmonic tori. Let ${\mathbf E}_p=f^{*}_pT(\S^{3}\times \S^{1})$ be the pullback bundle over $\T^2_p$. Let $\nabla$ denote the pullback Riemannian connection of $g_\epsilon$. For convenience, we denote $g_{\epsilon}$ by $g$.
The second derivative of energy at a critical point $f_p$ along a variation field $V$ is given (cf. \cite[(2.1)]{MM}, \cite{Sc}) by the index form
\begin{equation*}
I(V, V)=\int_{\T_p^{2}}\left\{\big|\nabla_{\frac{\partial }{\partial x_{0}}}V\big|^{2}+\big|\nabla_{\frac{\partial }{\partial x_{2}}}V\big|^{2}-R\left(\frac{\partial f_{P}}{\partial x_{0}}, V,\frac{\partial f_{p}}{\partial x_{0}}, V\right)-R\left(\frac{\partial f_{p}}{\partial x_{2}}, V,\frac{\partial f_{p}}{\partial x_{2}}, V\right)\right\}dx_{0}dx_{2}
\end{equation*}
where the Riemannian metric on $\T^2_p$ is compatible with the conformal structure $c_p$. Here we regard $\frac{\partial f_{p}}{\partial x_{0}}$ as a section of $T(\S^{3}\times \S^{1})$ defined by  $\frac{\partial f_{p}}{\partial x_{0}}(p)=(f_{p})_{*}(\frac{\partial }{\partial x_{0}}|_{p})$.

Since the tangent bundle of $\T^{2}$ is trivial, $\frac{\partial}{\partial x_{0}}, \frac{\partial }{\partial x_{2}}$ are global vector fields on $\T^{2}.$ In addition, when $p$ is in a neighbourhood of $p_{0}$, $f_{p}(\T^{2})\subset V_{00}$. Note that the tangent bundle of  $ V_{00}$ is trivial, we have $\mathbf E_p$ is a trivial vector bundle on $\T^{2}$ and $\frac{\partial }{\partial x_{i}}, i=0,1,2,3$ are global sections of $\mathbf E_p$.

For any smooth section $V$ of $\mathbf E_p$, we may write $V=\sum_{i=0}^{3}a_{i}\frac{\partial}{\partial x_{i}}$ for some functions $a_i\in C^\infty(\T^{2})$.   Associated to $V$, let
$$
V_{1}=\displaystyle\sum_{i}a_{i}\nabla_{\frac{\partial}{\partial x_{0}}}\frac{\partial}{\partial x_{i}}, \ \  \  \  \
V_{2}=\displaystyle\sum_{i}\frac{\partial a_{i}}{\partial x_{0}}\frac{\partial}{\partial x_{i}},
$$
$$
V_{3}=\displaystyle\sum_{i}a_{i}\nabla_{\frac{\partial}{\partial x_{2}}}\frac{\partial}{\partial x_{i}}, \ \  \  \  \
V_{4}=\displaystyle\sum_{i}\frac{\partial a_{i}}{\partial x_{2}}\frac{\partial}{\partial x_{i}}.
$$
Then we have
\begin{equation}\label{cd}
\nabla_{\frac{\partial }{\partial x_{0}}}V=V_{1}+V_{2}, \ \ \  \ \ \nabla_{\frac{\partial }{\partial x_{2}}}V=V_{3}+V_{4}.
\end{equation}
As $\frac{\partial f_p}{\partial x_0}=\frac{\partial}{\partial x_0}$ for the inclusion $f_p$, by expansion,
\begin{equation}\label{curvature 1 in SV}
R\left(\frac{\partial f_p}{\partial x_{0}}, V, \frac{\partial f_p}{\partial x_{0}}, V\right)=\sum_{i,j=0}^3a_{i}a_{j}\,R_{0i0j}.
\end{equation}
	We will show that for $i=0,2$, $j=0,\cdots, 3,$
\begin{equation}\label{curvature of torus}
R_{0i0j}=0, \ \ 	R_{0i2j}=0, \ \ 	R_{2i0j}=0,  \ \	R_{2i2j}=0.
\end{equation}
By \eqref{christoeffl}, we have
\begin{equation*} \begin{split}
R_{0i0j}&=-\frac{1}{2}\left(g_{00,ij}+g_{ij,00}-g_{i0,j0}-g_{j0,0i}\right)-\Gamma^{p}_{00}\Gamma^{q}_{ij}g_{pq}+\Gamma^{p}_{i0}\Gamma^{q}_{j0}g_{pq}\\&=\Gamma^{p}_{i0}\Gamma^{q}_{j0}g_{pq}=0.\\
\end{split}
\end{equation*}
Similarly, $R_{0i2j}=R_{2i0j}=R_{2i2j}=0$. It then follows from \eqref{section curvature} and \eqref{curvature 1 in SV} that
\begin{equation}\label{curvature in SV}
\begin{split}
R\left(\frac{\partial f_{p}}{\partial x_{0}}, V, \frac{\partial f_{p}}{\partial x_{0}}, V\right)=&a_{1}^{2}R_{0101}+a_{3}^{2}R_{0303}+2a_{1}a_{3}R_{0103}\\
=&a_{1}^{2}\Gamma_{01}^{p}\Gamma_{01}^{q}g_{pq}+a_{3}^{2}\Gamma_{03}^{p}\Gamma_{03}^{q}g_{pq}+2a_1a_3\Gamma^p_{01}\Gamma^q_{03}g_{pq}.
\end{split}
\end{equation}

On the other hand, by \eqref{christoeffl}, we have
\begin{equation*}
V_{1}=a_{i}\nabla_{\frac{\partial}{\partial x_{0}}}\frac{\partial}{\partial x_{i}}=\left(a_{1}\Gamma_{01}^{p}+a_{3}\Gamma_{03}^{p}\right)\frac{\partial}{\partial x_{p}}.
\end{equation*}
By \eqref{curvature in SV}, we obtain
 \begin{equation}\label{S2}
\big|V_{1}\big|^{2}_{g}=a_{1}^{2}\Gamma_{01}^{p}\Gamma_{01}^{q}g_{pq}+a_{3}^{2}\Gamma_{03}^{p}\Gamma_{03}^{q}g_{pq}+2a_{1}a_{3}\Gamma_{01}^{p}\Gamma_{03}^{q}g_{pq}=R\left(\frac{\partial f_{p}}{\partial x_{0}}, V, \frac{\partial f_{p}}{\partial x_{0}}, V\right).
\end{equation}
By \eqref{cd}, we have
\begin{equation}\label{S1}
\big|\nabla_{\frac{\partial }{\partial x_{0}}}V\big|^{2}_{g}=\big|V_{1}|^{2}_{g}+|V_{2}\big|^{2}_{g}+2g(V_{1}, V_{2}).
\end{equation}

We handle the third term above as
\begin{equation*}
\begin{split}
g(V_{1}, V_{2})=&\sum_{i=0}^{3}g\left(\frac{\partial a_{i}}{\partial x_{0}}\frac{\partial }{\partial x_{i}},a_{1}\nabla_{\frac{\partial }{\partial x_{0}}}
\frac{\partial}{\partial x_{1}}+a_{3}\nabla_{\frac{\partial }{\partial x_{0}}}
\frac{\partial}{\partial x_{3}} \right)\\
=&\sum_{i=0,2 ,j=1,3}a_j\frac{\partial a_{i}}{\partial x_{0}}\,g\left(\frac{\partial }{\partial x_{i}},\nabla_{\frac{\partial }{\partial x_{0}}}
\frac{\partial}{\partial x_{j}}\right)+\sum_{ i, j=1,3}a_j\frac{\partial a_{i}}{\partial x_{0}}\,g\left(\frac{\partial }{\partial x_{i}},\nabla_{\frac{\partial }{\partial x_{0}}}
\frac{\partial}{\partial x_{j}}\right).\\
\end{split}
\end{equation*}
By $\frac{\partial }{\partial x_{0}}g_{ij}=0$ and \eqref{christoeffl},
\begin{equation*}
\begin{split}
\sum_{i=0,2;j=1,3}a_j\frac{\partial a_{i}}{\partial x_{0}}\,g\left(\frac{\partial }{\partial x_{i}},\nabla_{\frac{\partial }{\partial x_{0}}}
\frac{\partial}{\partial x_{j}}\right)=&\sum_{i=0,2;j=1,3}a_j\frac{\partial a_{i}}{\partial x_{0}}\frac{\partial }{\partial x_{0}}g\left(\frac{\partial }{\partial x_{i}},
\frac{\partial}{\partial x_{j}}\right)\\
&-\sum_{i=0,2; j=1,3}a_j\frac{\partial a_{i}}{\partial x_{0}}\,g\left(\nabla_{\frac{\partial }{\partial x_{0}}}\frac{\partial }{\partial x_{i}},\frac{\partial}{\partial x_{j}}\right)=0
\end{split}
\end{equation*}
and
$$
g\left(\frac{\partial }{\partial x_{i}},\nabla_{\frac{\partial }{\partial x_{0}}}
\frac{\partial}{\partial x_{j}}\right)+g\left(\nabla_{\frac{\partial}{\partial x_0}}\frac{\partial }{\partial x_{i}},
\frac{\partial}{\partial x_{j}}\right)=\frac{\partial }{\partial x_{0}}g\left(\frac{\partial }{\partial x_{i}},
\frac{\partial}{\partial x_{j}}\right)=0
$$
in particular
$$
g\left(\frac{\partial }{\partial x_{i}},\nabla_{\frac{\partial }{\partial x_{0}}}
\frac{\partial}{\partial x_{i}}\right)=\frac{1}{2}\frac{\partial }{\partial x_{0}}g\left(\frac{\partial }{\partial x_{i}},
\frac{\partial}{\partial x_{i}}\right)=0, \ \ \ \ i=1,3.
$$
Therefore,
\begin{equation}\label{S3}
\begin{split}
g(V_{1}, V_{2})=&a_3\frac{\partial a_{1}}{\partial x_{0}}\,g\left(\frac{\partial }{\partial x_{1}},\nabla_{\frac{\partial }{\partial x_{0}}}
\frac{\partial}{\partial x_{3}}\right)+a_1\frac{\partial a_{3}}{\partial x_{0}} \,g\left(\frac{\partial }{\partial x_{3}},\nabla_{\frac{\partial }{\partial x_{0}}}
\frac{\partial}{\partial x_{1}}\right)\\
=&\left(a_3\frac{\partial a_{1}}{\partial x_{0}}-a_1\frac{\partial a_{3}}{\partial x_{0}}\right)\,g\left(\nabla_{\frac{\partial }{\partial x_{0}}}\frac{\partial }{\partial x_{3}},\frac{\partial}{\partial x_{1}}\right).\\
\end{split}
\end{equation}
In conclusion, by \eqref{S2}, \eqref{S1} and \eqref{S3}, we have
\begin{equation}\label{SV2}
\big|\nabla_{\frac{\partial }{\partial x_{0}}}V\big|^{2}_{g}-R\left(\frac{\partial f_{p}}{\partial x_{0}}, V, \frac{\partial f_{p}}{\partial x_{0}}, V\right)=\left|V_{2}\right|_{g}^{2}+2\left(a_3\frac{\partial a_{1}}{\partial x_{0}}-a_1\frac{\partial a_{3}}{\partial x_{0}}\right)g\left(\nabla_{\frac{\partial }{\partial x_{0}}}\frac{\partial }{\partial x_{3}}, \frac{\partial}{\partial x_{1}}\right).
\end{equation}
By \eqref{derivative of metric}, we get
\begin{equation}\label{stability 1}
g\left(\frac{\partial }{\partial x_{1}},\nabla_{\frac{\partial }{\partial x_{0}}}\frac{\partial}{\partial x_{3}}\right)=\Gamma_{03}^{p}g_{p1}=\frac{1}{2}(g_{01,3}-g_{03,1})=(1+\epsilon^{2})\frac{\pi}{4A^{2}}\left(x_{3}^{2}-x_{1}^{2}\right).
\end{equation}

Therefore, we have
\begin{equation}\label{stable 1}
\begin{split}
&\int_{\T_{p}^{2}}\left\{\big|\nabla_{\frac{\partial }{\partial x_{0}}}V\big|^{2}_{g}-R\left(\frac{\partial f_{p}}{\partial x_{0}}, V, \frac{\partial f_{p}}{\partial x_{0}}, V\right)\right\}\,dx_{0}dx_{2} \\
=&\int_{\T_{p}^2}\left\{ \big|V_2\big|^2_g+\left(4a_3\frac{\partial a_1}{\partial x_0} -2\frac{\partial (a_1a_3)}{\partial x_0}\right)
g\left(\frac{\partial }{\partial x_{1}},\nabla_{\frac{\partial }{\partial x_{0}}}\frac{\partial}{\partial x_{3}}\right)\right\}dx_{0}dx_{2}\\
=&\int_{\T_{p}^2}\left\{ \big|V_2\big|^2_g+4a_3\frac{\partial a_1}{\partial x_0}g\left(\frac{\partial }{\partial x_{1}},\nabla_{\frac{\partial }{\partial x_{0}}}\frac{\partial}{\partial x_{3}}\right) -2(1+\epsilon^{2})\frac{\partial}{\partial x_0} \left(a_1a_3\frac{\pi}{4A^2}(x^2_3-x^2_1)\right)
\right\}dx_{0}dx_{2}\\
=&\int_{\T_{p}^{2}}\left|V_{2}\right|_{g}^{2}+4a_3\frac{\partial a_{1} }{\partial x_{0}}g\left(\frac{\partial }{\partial x_{1}},\nabla_{\frac{\partial }{\partial x_{0}}}\frac{\partial}{\partial x_{3}}\right)\,dx_{0}dx_{2}
\end{split}
\end{equation}
where we have used \eqref{stability 1} and integration by parts.

There is a uniform positive constant $C$ arising from the largest eigenvalue of the symmetric matrix $(g_{ij})$ for $|x_{1}|,|x_{3}|\leq 1$ such that 
\[
\int_{\T_{p}^{2}}\big|V_{2}\big|_{g}^{2}dx_{0}dx_{2}+\int_{\T_{p}^{2}}\big|V_{4}\big|_{g}^{2}dx_{0}dx_{2}\geq \frac{1}{C}\int_{\T_{p}^{2}}\left(\left|\frac{\partial a_{1} }{\partial x_{0}}\right|^{2}+\left|\frac{\partial a_{3} }{\partial x_{0}}\right|^{2}+\left|\frac{\partial a_{1} }{\partial x_{2}}\right|^{2}+\left|\frac{\partial a_{3} }{\partial x_{2}}\right|^{2}\right)dx_{0}dx_{2}.
\]

Note that $$\int_{\T_{p}^{2}}\frac{\partial a_{1} }{\partial x_{0}}dx_0dx_2=0.$$
By Poincar\'e's inequality and \eqref{stability 1}, if $|x_{1}|, |x_{3}|$ are small enough, we have
\begin{equation}\label{SV 2}
\begin{split}
4\left|\int_{\T_{p}^{2}}\right.&\left. a_3\frac{\partial a_{1} }{\partial x_{0}} g\left(\frac{\partial }{\partial x_{1}},\nabla_{\frac{\partial }{\partial x_{0}}}\frac{\partial}{\partial x_{3}}\right)dx_{0}dx_{2}\right|\\
=&4\left|g\left(\frac{\partial }{\partial x_{1}},\nabla_{\frac{\partial }{\partial x_{0}}}\frac{\partial}{\partial x_{3}}\right)\int_{\T_{p}^{2}}\frac{\partial a_{1} }{\partial x_{0}}\left(a_{3}-\frac{1}{\mbox{Vol}(\T_{P}^{2})}\int_{\T_{p}^{2}}a_{3}\right)dx_{0}dx_{2}\right|\\
\leq& 2\left|g\left(\frac{\partial }{\partial x_{1}},\nabla_{\frac{\partial }{\partial x_{0}}}\frac{\partial}{\partial x_{3}}\right)\right|\left(\int_{\T_{p}^{2}}\left|\frac{\partial a_{1} }{\partial x_{0}}\right|^{2}dx_{0}dx_{2} +\int_{\T_{p}^{2}}\left|a_{3}-\frac{1}{\mbox{Vol}(\T_{P}^{2})}\int_{\T_{p}^{2}}a_{3} \right|^{2}dx_{0}dx_{2}\right)\\
\leq& C_{1}\left|\frac{\pi}{4A^{2}}\left(x_{3}^{2}-x_{1}^{2}\right)\right|\left(\int_{\T_{P}^{2}}\left|\frac{\partial a_{1} }{\partial x_{0}}\right|^{2}dx_{0}dx_{2} +\int_{\T_{p}^{2}}\left\{\left|\frac{\partial a_{3} }{\partial x_{0}}\right|^{2}+\left|\frac{\partial a_{3} }{\partial x_{2}}\right|^{2}\right\}dx_{0}dx_{2}\right)\\
\leq& \frac{1}{2}\left(\int_{\T_{p}^{2}}\big|V_{2}\big|_{g}^{2}dx_{0}dx_{2}+\int_{\T_{p}^{2}}\big|V_{4}\big|_{g}^{2}dx_{0}dx_{2}\right),
\end{split}
\end{equation}
where $C_{1}$ is  a uniform constant.

\vspace{.1cm}

Similar to \eqref{stable 1}, we have
\begin{equation}\label{stable 2}
\begin{split}
\int_{\T_{p}^{2}}&\left\{\big|\nabla_{\frac{\partial }{\partial x_{2}}}V\big|^{2}_{g}-R\left(\frac{\partial f_{p}}{\partial x_{2}}, V, \frac{\partial f_{p}}{\partial x_{2}}, V\right)\right\}dx_{0}dx_{2}\\
&=\int_{\T_{p}^{2}}\left\{\big|V_{4}\big|_{g}^{2}+4a_3\frac{\partial a_{1} }{\partial x_{2}}g\left(\frac{\partial }{\partial x_{1}},\nabla_{\frac{\partial }{\partial x_{2}}}\frac{\partial}{\partial x_{3}}\right)\right\}dx_{0}dx_{2}.
\end{split}
\end{equation}
By \eqref{derivative of metric}, we see that
\[
g\left(\frac{\partial }{\partial x_{1}},\nabla_{\frac{\partial }{\partial x_{2}}}\frac{\partial}{\partial x_{3}}\right)=g_{1p}\Gamma_{23}^{p}=\frac{1}{2}\left(g_{21,3}-g_{32,1}\right)=-\left(1+\epsilon^{2}\right)\frac{\pi}{ 2A^2}x_{1}x_{3}.
\]
When $|x_{1}|, |x_{3}|$ is small enough, as argued above, we have
\begin{equation*}
4\left|\int_{\T_{p}^{2}}a_3\frac{\partial a_{1} }{\partial x_{2}}g\left(\frac{\partial }{\partial x_{1}},\nabla_{\frac{\partial }{\partial x_{2}}}\frac{\partial}{\partial x_{3}}\right)dx_{0}dx_{2}\right|
\leq \frac{1}{2}\left(\int_{\T_{p}^{2}}\big|V_{2}\big|_{g}^{2}dx_{0}dx_{2}+\int_{\T_{p}^{2}}\big|V_{4}\big|_{g}^{2}dx_{0}dx_{2}\right).
\end{equation*}

\vspace{.1cm}
Combining with \eqref{stable 1}, \eqref{SV 2} and \eqref{stable 2},  we conclude $I(V, V)\geq 0$. Hence $f_p$ is a stable harmonic map when $p$ is in some neighbourhood of $p_0$ or $p_1$.
\end{proof}

\begin{theo}\label{unstable p}
The harmonic map $f_{p}$ is unstable, when $x_{1}=0$ and $|x_{3}|>\hspace{-.15cm}\sqrt{2}$.
\end{theo}
\begin{proof}
	Choose $V=a_{1}\frac{\partial}{\partial x_{1}}+a_{3}\frac{\partial}{\partial x_{3}}$ with $a_{1}=\frac{1}{2\pi}\cos(2\pi x_{0}), a_{3}=\frac{1}{2\pi}\sin(2\pi x_{0}).$ Then, by \eqref{SV2}, \eqref{derivative of metric} and \eqref{metric real}, when $x_{1}=0$
	\begin{equation*}
	\begin{split}
	\big|\nabla_{\frac{\partial }{\partial x_{0}}}V\big|^{2}_{g}-&R\left(\frac{\partial f_{p}}{\partial x_{0}}, V, \frac{\partial f_{p}}{\partial x_{0}}, V\right)= \big|V_2\big|^2_g -\frac{1}{\pi} g_{1p}\Gamma^p_{03}\\
	=& \sin^2(2\pi x_0) \,g_{11}+\cos^2(2\pi x_0) \,g_{33}-\frac{1}{2\pi}\left(g_{01,3}-g_{03,1}\right)\ \ \ \ \ \ \ \\
	=&\frac{1}{2}\left(\frac{1}{4A}+\frac{3}{4A^{2}}+\frac{\epsilon^{2}|w_{1}|^{2}}{4A^{2}}\right)-\left(1+\epsilon^{2}\right)\frac{1}{4A^{2}}x_{3}^{2}\\
	=&\frac{1}{2}\left(\frac{1+\epsilon^{2}}{4A}+\frac{3}{4A^{2}}-\frac{\epsilon^{2}}{4A^{2}}\right)-\left(1+\epsilon^{2}\right)\frac{1}{4A}+\left(1+\epsilon^{2}\right)\frac{1}{4A^{2}}\\
	=&\left(1+\epsilon^{2}\right)\left(\frac{1}{8}-\frac{1}{4}\right)\frac{1}{A}+\left(\frac{3}{8}+\frac{1}{4}+\frac{\epsilon^{2}}{8}\right)\frac{1}{A^{2}}\\
	< &-\frac{1}{4A}+\frac{3}{4A^{2}}<0
	\end{split}
	\end{equation*}
when $|x_3|>\hspace{-.15cm}\sqrt{2}$. 

\vspace{.1cm}

Note $\frac{\partial a_{1}}{\partial x_{2}}=\frac{\partial a_{3}}{\partial x_{2}}=0$. Hence  $V_{4}=0$.
By \eqref{stable 2},
\begin{equation*}
\int_{\T_{p}^{2}}\left\{|\nabla_{\frac{\partial }{\partial x_{2}}}V|^{2}_{g}-R\left(\frac{\partial f_{p}}{\partial x_{2}}, V, \frac{\partial f_{p}}{\partial x_{2}}, V\right)\right\}dx_{0}dx_{2}=0.
\end{equation*}
It follows $I(W, W)<0.$
\end{proof}

Let ${\mathbf E}^{\mathbb{C}}_p=f^{*}_pT(\S^{3}\times \S^{1})\otimes \mathbb{C}$ be the complexified pullback bundle over $\T^2_p$.
The pullback metric $f^*_pg_\epsilon$ makes ${\mathbf E}^{\mathbb{C}}_p$ a Hermitian bundle. Let $\nabla$ denote the pullback Riemannian connection of $g_\epsilon$ extended to a complex connection on ${\mathbf E}^{\mathbb{C}}_p$, which decomposes into $\nabla=\nabla'+\nabla''$ where
$$
\nabla':\mathcal A^{0,0}({\mathbf E}^{\mathbb{C}}_p)\to\mathcal A^{1,0}({\mathbf E}^{\mathbb{C}}_p),\,\,\,\,\,\,\,\nabla'':\mathcal A^{0,0}({\mathbf E}^{\mathbb{C}}_p)\to\mathcal A^{0,1}({\mathbf E}^{\mathbb{C}}_p)
$$
and $\mathcal A^{r,s}({\mathbf E}^{\mathbb{C}}_p)$ is the space of ${\mathbf E}^{\mathbb{C}}_p$-valued $(r,s)$-forms on $\T^2_p$.
The curvature 2-form is of type $(1,1)$ as the base $\T^2_p$ is of complex dimension 1. 
It is well known from \cite{KM}, \cite[Theorem 5.1]{AHS} and \cite[Proposition 1.3.7]{Ko} that there is a unique holomorphic structure $\bar{\partial}$ on ${\mathbf E}^{\mathbb{C}}_p$ so that $\nabla''=\bar{\partial}$ and with respect to which a section $W$ of $\mathbf E_p^\C$ is holomorphic if and only if
$$\nabla_{\frac{\partial}{\partial \overline w_0}} W =0.
$$
The discussion above is contained in \cite{MM}. 

\vspace{.1cm}

Let $H_{\overline{\partial}}^{0}(\T_p^{2}, \mathbf{E}_{p}^{\mathbb{C}})$ be the linear space of holomorphic sections of ${\mathbf E}^{\mathbb C}_p$.

\begin{theo}
When $p\neq p_{0},p_{1}$ is in $U_0\cup U_1$ where the neighbourhoods $U_{0}$, $U_1$ of $p_0,p_1$ are defined as in Theorem \ref{harmonic torus}, $H_{\overline{\partial}}^{0}(\T_p^{2}, \mathbf{E}_{p}^{\mathbb{C}})=\text{Span}_\C \left\{\frac{\partial}{\partial x_{0}}, \frac{\partial}{\partial x_{2}}\right\}$.
At $p_{0}, p_{1}$, we have $H_{\overline{\partial}}^{0}(\T_p^{2}, \mathbf{E}_p^{\mathbb{C}})=\text{Span}_\C \left\{\frac{\partial}{\partial x_{0}}, \frac{\partial}{\partial x_{2}}, \frac{\partial}{\partial x_{1}}, \frac{\partial}{\partial x_{3}}\right\}.$
\end{theo}
\begin{proof} Assume that $p$ is in the neighbourhood $U_0$ of $p_0$. So $f_p(\T^2_p)\subset V_{00}$. As explained in the proof of Theorem \ref{harmonic torus}, $\frac{\partial}{\partial x_0}, \cdots, \frac{\partial}{\partial x_3}$ are global sections of ${\mathbf E}_p$ hence they are also global sections of $\mathbf E_p^\C$. 
First, we show $R(\frac{\partial f_{P}}{\partial w_{0}}, W, \overline{\frac{\partial f_{p}}{\partial {w}_{0}}}, \overline{W})\geq 0$  for any smooth section $W\in\Gamma({\mathbf E}^{\mathbb C}_p)$. Assume $W=\sum^3_{i=0}a_{i}\frac{\partial }{\partial x_{i}}$ where $a_i\in C^\infty(\T_{p}^{2},\mathbb C)$. By Lemma \ref{curvature}, \eqref{curvature of torus} and \eqref{derivative of metric},
\begin{equation}\label{group 1}
\begin{split}
R\left(\frac{\partial f_{p}}{\partial w_{0}}, W, \overline{\frac{\partial f_{p}}{\partial {w}_{0}}}, \overline{W}\right)=&\,|a_{1}|^{2}R_{0101}+|a_{3}|^{2}R_{0303}+\i\left(|a_{1}|^{2}R_{0121}+|a_{3}|^{2}R_{0323}\right)\\
&-\i\left(|a_{1}|^{2}R_{2101}+|a_{3}|^{2}R_{2303}\right)+|a_{1}|^{2}R_{2121}+|a_{3}|^{2}R_{2323}\\
=&\,|a_{1}|^{2}R_{0101}+|a_{3}|^{2}R_{0303}+|a_{1}|^{2}R_{2121}+|a_{3}|^{2}R_{2323}\\
=&\,\frac{1}{4}(|a_{1}|^{2}+|a_{3}|^{2})\left(g^{33}(g_{03,1}-g_{01,3})^{2}+g^{11}(g_{21,3}-g_{32,1})^{2}\right)\\
=&\,\frac{1}{4}(1+\epsilon^{2})^{2}(|a_{1}|^{2}+|a_{3}|^{2})g^{33}\left(\left(\frac{\pi}{2A^{2}}(x_{1}^{2}-x_{3}^{2})\right)^{2}+\left(\frac{\pi}{A^{2}}x_{1}x_{3}\right)^{2}\right)\geq 0.
\end{split}
\end{equation}

The second derivative of energy at a critical point $f_p$ along $W$ is given in \cite[(2.3)]{MM} by the index form
\begin{equation*}
I(W, W)= 2\ii\int_{\T_{p}^{2}}\left\{\left|\nabla_{\frac{\partial }{\partial \overline{w}_0}}W\right|^{2}- R\left(W,\frac{\partial f_p}{\partial w_0}, \overline{W}, \overline{\frac{\partial f_{p} }{\partial w_0}}\right)\right\} d w_0\wedge d\overline{w}_0.
\end{equation*}
Assume $W$ is holomorphic. Since $f_{p}$ is stable by Theorem \ref{harmonic torus},
\begin{equation*}
I(W, W)= 2\ii\int_{\T_{p}^{2}}\left\{- R\left(W,\frac{\partial f_p}{\partial w_0}, \overline{W}, \overline{\frac{\partial f_p}{\partial w_0}}\right)\right\} d w_0\wedge d\overline{w}_0\geq0.
\end{equation*}
By \eqref{group 1}, we have
$$
R\left(W,\frac{\partial f_p}{\partial w_0}, \overline{W}, \overline{\frac{\partial f_p}{\partial w_0}}\right)\equiv 0.
$$
When $p\neq p_{0}, p_{1}$, we have $a_{1}=a_{3}=0.$
Then
\begin{equation*}
0=\nabla_{\frac{\partial}{\partial \overline{w}_{0}}}W=\frac{\partial a_{0}}{\partial \overline{w}_{0}}\frac{\partial}{\partial x_{0}}+\frac{\partial a_{2}}{\partial \overline{w}_{0}}\frac{\partial}{\partial x_{2}},
\end{equation*}
as $\Gamma^k_{00},\Gamma^k_{02},\Gamma^k_{22}$ are all 0 by \eqref{christoeffl}, this implies $a_{0}, a_{2}$ are constant.

At $p_{0}$, we have $\Gamma_{ij}^{k}=0, i=0,2$ and all $j$. In fact,  since $g_{ij}$ is independent of $x_{0},x_{2}$ and $g_{ts}$ with $t,s=0,2$ are constant
$$
\Gamma_{01}^{k}(p_{0})=\frac{1}{2}g^{kq}\left(g_{q0,1}-g_{01,q}\right)(p_0)=\frac{1}{2}g^{k3}\big(g_{30,1}-g_{01,3})(p_{0}\big)=0
$$
by \eqref{derivative of metric}; the Christoffel symbols vanish for other $i, j$ similarly. So at $p_0$
$$
\nabla_{\frac{\partial}{\partial x_{i}}}\frac{\partial }{\partial x_{j}}=\sum^3_{k=0}\Gamma_{ij}^{k}\frac{\partial}{\partial x_{k}}=0
$$
for $ i=0,2, j=0,1,2,3$.
It follows
\begin{equation}\label{vanish}
\nabla_{\frac{\partial}{\partial \overline{w}_{0}}}\frac{\partial }{\partial x_{j}}=0,
\end{equation}
 i.e. $\frac{\partial }{\partial x_{j}}$ are holomorphic sections. Thus $\text{Span}_\C\{\frac{\partial}{\partial x_i}: 0\leq i\leq 3\}\subseteq H_{\overline{\partial}}^{0}(\T_p^{2}, \mathbf{E}_{p}^{\mathbb{C}})$. On the other hand, for any $W\in H_{\overline{\partial}}^{0}(\T_p^{2}, \mathbf{E}_{p}^{\mathbb{C}})$
 $$
 0=\nabla_{\frac{\partial}{\partial \overline{w}_{0}}}W=\sum^3_{i=0}\frac{\partial a_{i}}{\partial \overline{w}_{0}}\frac{\partial}{\partial x_{i}}
 $$
 by \eqref{vanish}. So $a_i$ are holomorphic functions on $\T^2_p$ and they must be constant. In turn, $W\in \text{Span}_\C\{\frac{\partial}{\partial x_i}: 0\leq i\leq 3\}$.
\end{proof}

\begin{theo}\label{minimal surface}
Each fibre $\T^2_p$ of the Hopf fibration is a flat and totally geodesic minimal surface in $(\S^3\times\S^1,{g}_\epsilon)$. It is a stable minimal surface if $p=p_0,p_1$ and it is unstable if $p$ is as in Theorem \ref{unstable p}.
\end{theo}
\begin{proof}
Reading from \eqref{metric real}, ${g}_{\epsilon,00}={g}_{\epsilon,22}=(1+\epsilon^2)\pi$ and ${g}_{\epsilon,02}={g}_{\epsilon,20}=0$. It follows that $\T^2_p$ is flat for the induced metric from ${g}_\epsilon$. The Gauss formula together with \eqref{christoeffl} asserts that $\T^2_p$ is totally geodesic in $(\S^3\times\S^1,{g}_\epsilon)$.

According to a theorem of Ejiri and Micallef in \cite{EM}, the Morse index $i_A$ of the area functional for the minimal surface $\T^2_p$ is no smaller than the Morse index $i_E$ of the harmonic map $f_p$. So $\T^2_p$ is unstable as a minimal surface if $f_p$ is unstable as a harmonic map.

We now consider $p=p_0$.  Let $X$ be a normal vector field of $\T^{2}_{p_0}$ and $\nabla^{\perp}$ be the normal connection of $\T^{2}_{p_0}$. The second variation formula of area for the minimal surface $\T^2_{p_0}$ is
\begin{eqnarray}
\delta^{2} |\T^{2}_{p_0}|(X, X)&=&\int_{\T^{2}_p}\big|\nabla^{\perp}X\big|^{2}-\mbox{tr}_{g_\epsilon}R_{{g}_\epsilon}(\cdot, X,\cdot, X)-{g}_\epsilon(A, X)^{2}\\&=&\int_{\T^{2}_{p_0}}\big|\nabla^{\perp}X\big|^{2}-\mbox{tr}_{{g}_\epsilon}R_{{g}_\epsilon}(\cdot, X,\cdot, X) \nonumber
\end{eqnarray}
where $A\equiv 0$ is the second fundamental form of the totally geodesic $\T^2_{p_0}$ in $(\S^3\times\S^1,{g}_\epsilon)$.

	We claim that at $w_{1}=0$, i.e., at $p_0$,
\begin{equation}\label{curvature of torus 1}
R_{0i0j}=0, \ \ 	R_{0i2j}=0, \ \ 	R_{2i2j}=0,\ \ 0\leq i,j\leq 3.
\end{equation}
Reasoning similar as in the proof of Lemma \ref{curvature},
\begin{equation*} \begin{split}
R_{0i0j}&=-\frac{1}{2}\left(g_{00,ij}+g_{ij,00}-g_{i0,j0}-g_{j0,0i}\right)-\Gamma^{p}_{00}\Gamma^{q}_{ij}g_{pq}+\Gamma^{p}_{i0}\Gamma^{q}_{j0}g_{pq}\\&=\frac{1}{4}g^{ts}\left(g_{0t,i}-g_{0i,t}\right)\left(g_{0s,j}-g_{0j,s}\right)\\
\end{split}
\end{equation*}
and
\begin{equation*} \begin{split}
R_{0i2j}&=-\frac{1}{2}\left(g_{02,ij}+g_{ij,02}-g_{i2,j0}-g_{j0,2i}\right)-\Gamma^{p}_{02}\Gamma^{q}_{ij}g_{pq}+\Gamma^{p}_{i2}\Gamma^{q}_{j0}g_{pq}\\&=\frac{1}{4}g^{ts}(g_{2t,i}-g_{2i,t})(g_{0s,j}-g_{0j,s}).\\
\end{split}
\end{equation*}
By \eqref{derivative of metric} we see $g_{03,1}-g_{10,3}=0$ at $w_1=0$;  further, \eqref{metric real} is independent of $x_{0}$ and $x_{2}$ and the constancy of the related entries, we have $g_{0s,j}-g_{0s,j}=0$ for all $j,s$. Therefore, $R_{0i0j}=0$ and  $R_{0i2j}=0.$
Similarly,  $R_{2i2j}=0$.

Therefore, by \eqref{curvature of torus 1},  $\mbox{tr}_{{g}_\epsilon}R_{{g}_\epsilon}(\cdot, X,\cdot, X)=0$. Hence $\T^{2}_{p_0}$ is a stable minimal surface. Similarly, so is $\T^2_{p_1}$. \end{proof}

\section{Generalization to $\mathbb S^{2n-1}\times \mathbb S^{1}$} In this section, we show that Theorem \ref{main3} remains true for the Calabi-Eckmann complex $n$-manifold $\S^{2n-1}\times\S^1$.  However, our previous argument does not lead to a generalization of Theorem \ref{main4} to higher dimensions since it is not clear whether it is still the case that $R(\frac{\partial f_{p}}{\partial w_{0}}, W, \overline{\frac{\partial f_{p}}{\partial {w}_{0}}}, \overline{W})=0$ implies 
 $W=a_{0}\frac{\partial}{\partial x_{0}}+a_{n}\frac{\partial}{\partial x_{n}}$. 
 
\subsection{Hermitian structures on $\S^{2n-1}\times \S^{1}$.} Set 
$$
V_{\alpha0}=\left\{(z_0,\cdots,z_{n-1}, z_0'): (z_0,\cdots, z_{n-1})\in\S^{2n-1}\subset\C^n, (z'_0)\in\S^{1}\subset\C^{1}, z_{\alpha}z'_{0}\neq 0\right\}
$$
where $\alpha = 0,1, \cdots,n-1$. The family $\{V_{\alpha0}\}$ is an open cover of $\S^{2n-1}\times \S^{1}$. On $V_{\alpha0}$, with $\alpha, j=0,1, \cdots, n$, set 
\begin{eqnarray*}
w_{\alpha j }\ &=&\ \frac{z_{j}}{z_{\alpha}},\label{1}\,\,\,\,\,\,j\neq \alpha\\
t_{\alpha 0}\ &=&\ \frac{1}{2\pi\ii}\left(\log z_{\alpha}+\sqrt{-1}\log z'_{0}\right) \mod(1,\ii). 
\end{eqnarray*}

\begin{prop}(Calabi-Eckmann) \label{complex coordinates n}
	Each $V_{\alpha0}$ is homeomorphic to $\mathbb C^{n-1}\times \mathbb T^{2}$. On $U_{\alpha\beta}=\{(w_{\alpha j}, t_{\alpha0})\in\C^{n+1}: 0<\mathfrak{Re} \,t_{\alpha0},\mathfrak{Im}\, t_{\alpha0}<1\}\subset V_{\alpha0}$,  $(w_{\alpha j}, t_{\alpha0})$
	is a complex coordinate system of $V_{\alpha0}\subset\S^{2n-1}\times \S^{1}$. For this complex structure, the fibre bundle $\S^{2n-1}\times \S^{1}\to \mathbb CP^{n-1}$ is complex analytic and each fibre is a holomorphic nonsingular torus.
\end{prop}

On $U_{00}$ we set 
$
(w_0,w_1,\cdots,w_{n-1})=(t_{00}, w_{01},\cdots,w_{0(n-1)}).
$
The inclusion map $$\vp:\S^{2n-1}\times \S^{1}\rightarrow \C^n\times \C^1$$ is expressed as
\begin{equation}\label{3d map n}
\left\{\begin{aligned}
z_{0}\ &=\ A^{-\frac{1}{2}}e^{\ii\pi(w_{0}+\overline{w}_{0})},\\
z'_{0}\ &=\ e^{\pi(w_{0}-\overline{w}_{0})-\frac{1}{2}\ii\log A},\\
z_{i}\ &=\ z_{0}w_{i}=w_{i}A^{-\frac{1}{2}}e^{\ii\pi(w_{0}+\overline{w}_{0})},  i=1,\cdots,n-1,\\
\end{aligned}\right.
\end{equation}
where $A=1+\sum_{i=1}^{n-1}|w_{i}|^{2}$, and by straight computation 
\begin{equation*}
\small{
	\left\{\begin{aligned}
	\frac{\partial z_{0}}{\partial w_{0}}&=\ii \pi z_{0},  &\frac{\partial z'_{0}}{\partial w_{0}}&= \pi z'_{0},
	&\frac{\partial z_{i}}{\partial w_{0}}&=\ii \pi z_{i},\,\, i=1,\cdots,n-1 \\
	\frac{\partial z_{0}}{\partial w_{i}}&=-\frac{1}{2}\frac{\overline{w}_{i}z_{0}}{A}, &\frac{\partial z'_{0}}{\partial w_{i}}&=-\frac{\ii}{2}\frac{\overline{w}_{i}z'_{0}}{A},&\frac{\partial z_{i}}{\partial w_{i}}&=z_{0}-\frac{1}{2}\frac{\overline{w}_{i}z_{i}}{A},\\
	~&~ &~&~ &\frac{\partial z_{i}}{\partial w_{j}}&=-\frac{1}{2}\frac{\overline{w}_{j}z_{i}}{A},\,\,j\neq i.\end{aligned}\right.}
\end{equation*}

Take the Hermitian metrics on $\C^n\times \C^1$
$$
h_{\epsilon}=\displaystyle\sum_{i=0}^{n-1}dz_{i}\otimes d\overline{z}_{i}+\epsilon^{2}dz'_{0}\otimes d\overline{z}'_{0}, \ \ \ \epsilon\in [0,1].
$$
For any $X\in T^{1,0}(\S^{2n-1}\times\S^1)$, let $\vp_*(X)^{1,0}\in T^{1,0}(\C^n\times\C^1)$ be the $(1,0)$ part of the push forward $\vp_*X$. Then
$$
g_\epsilon(X,Y) = h_\epsilon\left(\vp_*(X)^{1,0},\vp_*(Y)^{1,0}\right),\,\,\,\,X,Y\in T^{1,0}(\S^{2n-1}\times\S^1)
$$
is a Hermitian metric on the complex manifold $\S^{2n-1}\times\S^1$. Its components
$$
{g}_{\epsilon, i\overline{j}}=h_{\epsilon}\left(\vp_{*}\left(\frac{\partial}{\partial w_{i}}\right)^{1,0}, \vp_{*}\left(\frac{\partial}{\partial w_{j}}\right)^{1,0}\right)
$$ 
are given by the Hermitian matrix
\begin{equation}\label{induced metric n}
\begin{aligned}
\big({g}_{\epsilon,i\overline{j}}\big)=\left(
\begin{matrix}
(1+\epsilon^{2})\pi^2& \ \ \frac{(1+\epsilon^{2})\ii \pi }{2}\frac{w_{1}}{A}  &  \cdots& \ \ \frac{(1+\epsilon^{2})\ii \pi }{2}\frac{w_{i}}{A} &\cdots&\ \ \frac{(1+\epsilon^{2})\ii \pi }{2}\frac{w_{n-1}}{A}\\[3mm]
- \frac{(1+\epsilon^{2})\ii\pi}{2}\frac{\overline{w}_{1}}{A}& \ \ \frac{1}{A}+\frac{(-3+\epsilon^{2})}{4}\frac{|w_{1}|^2}{A^2}&\cdots &\ \ \frac{(-3+\epsilon^{2})}{4} \frac{\overline{w}_{1}w_{i}}{A^{2}}&\cdots &\frac{(-3+\epsilon^{2})}{4} \frac{\overline{w}_{1}w_{n-1}}{A^{2}}\\[3mm]
\vdots&\vdots&\cdots&\vdots&\cdots&\vdots\\
-\frac{\ii(1+\epsilon^{2})\pi}{2}\frac{\overline{w}_{n-1}}{A}& \ \ \frac{(-3+\epsilon^{2})}{4} \frac{\overline{w}_{n-1}w_{1}}{A^{2}} &\cdots&\ \ \frac{(-3+\epsilon^{2})}{4} \frac{\overline{w}_{n-1}w_{i}}{A^{2}} &\cdots&\ \ \frac{1}{A}+\frac{(-3+\epsilon^{2})}{4}\frac{|w_{n-1}|^2}{A^2}\\
\end{matrix}\right).
\end{aligned}
\end{equation}
It is positive definite since $\det ({g}_{\epsilon,i\overline{j}})=\frac{1+\epsilon^{2}}{A^{n}}\pi^{2}. $
\subsection{The results}

Let ${\mathbf E}^{\mathbb{C}}_p=f^{*}_pT(\S^{2n-1}\times \S^{1})\otimes \mathbb{C}$ be the complexified pullback bundle over $\T^2_p$.
Denote $w_{i}=x_{i}+\ii x_{i+n}$. For a Hermitian metric $g$, locally written as $\mathcal A+\ii \mathcal B$, the underlying Riemannian metric is 
\begin{equation}\label{Rie}
\frac{1}{2}\begin{pmatrix}
\mathcal A&-\mathcal B\\
\mathcal B&\mathcal A
\end{pmatrix}.
\end{equation} 
Let $g_{ij}$ denote the components of the Riemannian metric arising from $g_\epsilon$. As in section 3, the following Riemannian geometric properties of $g$ are crucial in our study of stability.  
\begin{prop}\label{tool} In $V_{00}$, we have 
	\begin{enumerate}
		\item \label{christoeffel 3}
		$\Gamma_{kl}^{p}=0$, for $k, l=0,n$ and $ 0\leq p\leq 2n-1.$
		\item \label{order of first} 
		$
		|g_{0p,q}-g_{0q,p}|=O(|w'|)$ \ and \ $\label{order of first 3}
		|g_{np,q}-g_{nq,p}|=O(|w'|)$ 
		where $w'=(w_{1}, \cdots, w_{n-1})$, $0\leq p,q\leq 2n-1$.
		\item Moreover, for $0\leq r,s,p,q\leq 2n-1$, 
		\begin{equation*}
		R_{0r0s}=\Gamma^{p}_{r0}\Gamma^{q}_{s0}g_{pq},
		\end{equation*}	
			\begin{equation*}
		R_{0rns}=\Gamma^{p}_{rn}\Gamma^{q}_{s0}g_{pq},
		\end{equation*}	
		\begin{equation*}
		R_{nrns}=\Gamma^{p}_{rn}\Gamma^{q}_{sn}g_{pq}
		\end{equation*}	
		and
		$$
		R_{0n0r}=R_{n0nr}=0.
		$$
		\item \label{4} $R(\frac{\partial f_{p}}{\partial w_{0}}, W, \overline{\frac{\partial f_{p}}{\partial {w}_{0}}}, \overline{W})\geq 0$,  for any $W\in {\mathbf E}^{\mathbb{C}}_p$.
	\end{enumerate}
\end{prop}
\begin{proof}
	By \eqref{induced metric n} and \eqref{Rie}, $g$ is independent of $x_{0}$ and $x_{n}.$ 
	We have for $1\leq i\leq n-1$, 
	\begin{equation*}
	\begin{aligned}
	g_{00}&=\frac{1}{2}(1+\epsilon^{2})\pi^{2},\,g_{0n}=0,\, g_{0i}=-\frac{(1+\epsilon)^{2}\pi}{4}\frac{x_{i+n}}{A},\, g_{0(n+i)}=-\frac{(1+\epsilon)^{2}\pi}{4}\frac{x_{i}}{A};\\
	g_{n0}&=0,\ g_{nn}=\frac{1}{2}(1+\epsilon^{2})\pi^{2},\,  g_{ni}=\frac{(1+\epsilon)^{2}\pi}{4}\frac{x_{i}}{A}, \ g_{n(n+i)}=-\frac{(1+\epsilon)^{2}\pi}{4}\frac{x_{i+n}}{A}.\\
	\end{aligned}
	\end{equation*}
	Then \eqref{christoeffel 3} follows immediately: 
	\[\Gamma_{00}^{p}=\frac{1}{2}g^{pq}(2g_{0q,0}-g_{00,q})=0,\]
	since $g$ is independent of $x_{0}$ and  $x_{n}$.
	
	If $p=0,n$ or $q=0,n$, then \eqref{order of first} is obvious. 
	For $1\leq i,j\leq n-1,$
	\begin{equation}\label{derivative}
	\begin{split}
	&g_{0i,j}-g_{0j,i}=(1+\epsilon^{2})\frac{\pi}{2A^{2}}\left(x_{n+i}x_{j}-x_{n+j}x_{i}\right), \\
	&g_{0(n+i),(n+j)}-g_{0(n+j),(n+i)}=(1+\epsilon^{2})\frac{\pi}{2A^{2}}\left(x_{i}x_{n+j}-x_{j}x_{n+i}\right),\\
	&g_{0i,(n+j)}-g_{0(n+j),i}=(1+\epsilon^{2})\frac{\pi}{2A^{2}}\left(x_{n+i}x_{n+j}-x_{j}x_{i}\right).
	\end{split}
	\end{equation}
	Similarly, we have, for $ 1\leq i,j\leq n-1,$
	\begin{equation}\label{derivative 2}
	\begin{split}
	&g_{ni,j}-g_{nj,i}=0,\\
	&g_{n(n+i),(n+j)}-g_{n(n+j),(n+i)}=0,\\
	&g_{ni,(n+j)}-g_{n(n+j),i}=(1+\epsilon^{2})\frac{\pi}{2A^{2}}\left(-2x_{i}x_{n+j}\right).
	\end{split}
	\end{equation}
	Then $(2)$ follows.
	
	To prove (3), note that for $0\leq r,s \leq 2n-1,$	
	\begin{equation*} \begin{split}
	R_{0r0s}&=-\frac{1}{2}\left(g_{00,rs}+g_{rs,00}-g_{r0,s0}-g_{s0,0r}\right)-\Gamma^{p}_{00}\Gamma^{q}_{rs}g_{pq}+\Gamma^{p}_{r0}\Gamma^{q}_{s0}g_{pq}=\Gamma^{p}_{r0}\Gamma^{q}_{s0}g_{pq}\\
\end{split}
	\end{equation*}
	Similarly, $R_{nrns}=\Gamma^{p}_{rn}\Gamma^{q}_{sn}g_{pq}$ and $R_{0rns}=\Gamma^{p}_{rn}\Gamma^{q}_{s0}g_{pq}$.
	It follows from \eqref{christoeffel 3} that $R_{0i0r}=R_{ninr}=0$ when $i=0, n$ and $0\leq r\leq 2n-1.$ 
	
	\vspace{.1cm}
	
	To prove \eqref{4}, assume $W=\sum_{r=0}^{2n-1}a_{r}\frac{\partial}{\partial x_{r}},$ where $a_{r}$ are complex functions. 
	Then 
	\begin{equation*}
	\begin{split}
	R\left(\frac{\partial f_{p}}{\partial w_{0}}, W, \overline{\frac{\partial f_{p}}{\partial {w}_{0}}}, \overline{W}\right)&=\sum_{r, s=0}^{2n-1}a_{r}\bar{a}_{s}R_{0r0s}+\ii a_{r}\bar{a}_{s}R_{0rns}-\ii a_{r}\bar{a}_{s}R_{nr0s}+a_{r}\bar{a}_{s}R_{nrns}\\
	&=\sum_{r,s, p,q=0}^{2n-1}a_{r}\bar{a}_{s}g_{pq}(\Gamma^{p}_{0r}\Gamma^{q}_{0s}+\ii \Gamma^{p}_{0s}\Gamma^{q}_{nr}-\ii \Gamma^{p}_{ns}\Gamma^{q}_{0r}+\Gamma^{p}_{nr}\Gamma^{q}_{ns}).
	\end{split}
	\end{equation*} 
	Denote $b_{p}=\sum_{r=0}^{2n-1}\Gamma^{p}_{0r}a_{r}$ and $c_{p}=\sum_{r=0}^{2n-1}\Gamma^{p}_{nr}a_{r}$.
	Then \[R\left(\frac{\partial f_{p}}{\partial w_{0}}, W, \overline{\frac{\partial f_{p}}{\partial {w}_{0}}}, \overline{W}\right)=\sum_{p,q=0}^{2n-1}g_{pq}(b_{p}+\ii c_{p})(\bar{b}_{q}-\ii \bar{c}_{q})\geq 0\]	
	as required.
\end{proof}
Let $S_{i}=\overbrace{(0,\cdots, 0)}^{{}i-1}\times \mathbb S^{1} \times\overbrace{(0,\cdots, 0)}^{{}n-i}\subset\mathbb{C}^{n}$ and $T_{i}=S_{i}\times \mathbb S^{1}\subset \mathbb S^{2n-1}\times \mathbb S^{1}, 1\leq i\leq n.$

\begin{theo}\label{harmonic torus 3}
	Each $f_p$ is a harmonic map from $(\T^2_p,c_p)$ to $(\S^{2n-1}\times\S^1, {g}_\epsilon)$. There exist neighbourhoods of $T_{i}, 1\leq i\leq n,$ such that $f_p$ is a stable harmonic map when $p$ belongs to the neighbourhoods. 
\end{theo}
\begin{proof} Harmonicity of $f_p$ follows as in the proof of Theorem \ref{harmonic torus}. Next, we examine the stability of these harmonic tori. Let ${\mathbf E}_p=f^{*}_pT(\S^{2n-1}\times \S^{1})$ be the pullback bundle over $\T^2_p$. Let $\nabla$ denote the pullback Riemannian connection of $g_\epsilon$. For convenience, we denote $g_{\epsilon}$ by $g$.
	The second derivative of energy at $f_p$ along a variation field $V$ is 
	\begin{equation*}
	I(V, V)=\int_{\T_p^{2}}\left\{\big|\nabla_{\frac{\partial }{\partial x_{0}}}V\big|^{2}+\big|\nabla_{\frac{\partial }{\partial x_{n}}}V\big|^{2}-R\left(\frac{\partial f_{p}}{\partial x_{0}}, V,\frac{\partial f_{p}}{\partial x_{0}}, V\right)-R\left(\frac{\partial f_{p}}{\partial x_{n}}, V,\frac{\partial f_{p}}{\partial x_{n}}, V\right)\right\}dx_{0}dx_{n}.
	\end{equation*}

As in the proof of Theorem \ref{harmonic torus}, a smooth section of $\mathbf E_p$ can be written as $V=\sum_{i=0}^{2n-1}a_{i}\frac{\partial}{\partial x_{i}}$ for some functions $a_i\in C^\infty(\T^{2})$. Set 
	$V_{1}=a_{i}\nabla_{\frac{\partial}{\partial x_{0}}}\frac{\partial}{\partial x_{i}}, 
	V_{2}=\frac{\partial a_{i}}{\partial x_{0}}\frac{\partial}{\partial x_{i}},
	V_{3}=a_{i}\nabla_{\frac{\partial}{\partial x_{n}}}\frac{\partial}{\partial x_{i}},
	V_{4}=\frac{\partial a_{i}}{\partial x_{n}}\frac{\partial}{\partial x_{i}}.$
Then $\nabla_{\frac{\partial }{\partial x_{0}}}V=V_{1}+V_{2}$ and $\nabla_{\frac{\partial }{\partial x_{3}}}V=V_{3}+V_{4}.$ 

\vspace{.1cm}

Similar to \eqref{S1}, \eqref{S2} and  \eqref{S3}, we have
	\begin{equation}\label{S2 3}
	\begin{split}
\big|V_{1}\big|^{2}_{g}=&\sum_{i,j\neq 0,n}a_{i}a_{j}\Gamma_{0i}^{p}\Gamma_{0j}^{q}g_{pq}=R\left(\frac{\partial f_{p}}{\partial x_{0}}, V, \frac{\partial f_{p}}{\partial x_{0}}, V\right),\\
\big|\nabla_{\frac{\partial }{\partial x_{0}}}V\big|^{2}_{g}=&\ \big|V_{1}|^{2}_{g}+|V_{2}\big|^{2}_{g}+2g(V_{1}, V_{2}),\\
	g(V_{1}, V_{2})=&\sum_{i, j\neq 0,n;i\neq j}a_j\frac{\partial a_{i}}{\partial x_{0}}\,g\left(\frac{\partial }{\partial x_{i}},\nabla_{\frac{\partial }{\partial x_{0}}}
	\frac{\partial}{\partial x_{j}}\right).\\
	\end{split}
	\end{equation}
	In conclusion, 
	\begin{equation}\label{SV2 3}
	\big|\nabla_{\frac{\partial }{\partial x_{0}}}V\big|^{2}_{g}-R\left(\frac{\partial f_{p}}{\partial x_{0}}, V, \frac{\partial f_{p}}{\partial x_{0}}, V\right)=\left|V_{2}\right|_{g}^{2}+2\sum_{i, j\neq 0,n; i\neq j}a_j\frac{\partial a_{i}}{\partial x_{0}}\,g\left(\frac{\partial }{\partial x_{i}},\nabla_{\frac{\partial }{\partial x_{0}}}
	\frac{\partial}{\partial x_{j}}\right).
	\end{equation}
	Also note that $g\left(\frac{\partial }{\partial x_{i}},\nabla_{\frac{\partial }{\partial x_{0}}}\frac{\partial}{\partial x_{j}}\right)$ is independent of $x_{0}, x_{n}.$ 
	Therefore
	\begin{equation}\label{stable 1 3}
	\begin{split}
	&\int_{\T_{p}^{2}}\left\{\big|\nabla_{\frac{\partial }{\partial x_{0}}}V\big|^{2}_{g}-R\left(\frac{\partial f_{p}}{\partial x_{0}}, V, \frac{\partial f_{p}}{\partial x_{0}}, V\right)\right\}\,dx_{0}dx_{n} \\
	=&\int_{\T_{p}^{2}}\big|V_2\big|^2_gdx_{0}dx_{n}+2\sum_{i, j\neq 0,n; i\neq j}g\left(\frac{\partial }{\partial x_{i}},\nabla_{\frac{\partial }{\partial x_{0}}}
	\frac{\partial}{\partial x_{j}}\right)\int_{\T_{p}^{2}}a_j\frac{\partial a_{i}}{\partial x_{0}}\,dx_{0}dx_{n}.
	\end{split}
	\end{equation}
	
	There is a uniform positive constant $C$ arising from the smallest eigenvalue of the symmetric matrix $(g_{ij})$ for $|w'|\leq 1$ such that 
	\[
	\int_{\T_{p}^{2}}\big|V_{2}\big|_{g}^{2}dx_{0}dx_{n}+\int_{\T_{p}^{2}}\big|V_{4}\big|_{g}^{2}dx_{0}dx_{n}\geq \frac{1}{C}\int_{\T_{p}^{2}}\sum_{i\neq 0,n}\left|\frac{\partial a_{i} }{\partial x_{0}}\right|^{2} +\sum_{i\neq 0,n}\left|\frac{\partial a_{i }}{\partial x_{n}}\right|^{2}dx_{0}dx_{n}.
	\]
	
	Note that $\int_{\T_{p}^{2}}\frac{\partial a_{i} }{\partial x_{0}}dx_0dx_n=0.$
	By Poincar\'e's inequality and Proposition \ref{tool}, if $|w'|$ is small enough, we have
	\begin{equation}\label{SV 2 3}
	\begin{split}
	\left|\int_{\T_{p}^{2}}\right.&\left. a_j\frac{\partial a_{i} }{\partial x_{0}} g\left(\frac{\partial }{\partial x_{i}},\nabla_{\frac{\partial }{\partial x_{0}}}\frac{\partial}{\partial x_{j}}\right)dx_{0}dx_{n}\right|\\
	\leq& C_{1}|w'|\left(\int_{\T_{p}^{2}}\left|\frac{\partial a_{i} }{\partial x_{0}}\right|^{2}dx_{0}dx_{n} +\int_{\T_{p}^{2}}\left\{\left|\frac{\partial a_{j} }{\partial x_{0}}\right|^{2}+\left|\frac{\partial a_{j} }{\partial x_{n}}\right|^{2}\right\}dx_{0}dx_{n}\right)\\
	\leq& \frac{1}{4}\left(\int_{\T_{p}^{2}}\big|V_{2}\big|_{g}^{2}dx_{0}dx_{n}+\int_{\T_{p}^{2}}\big|V_{4}\big|_{g}^{2}dx_{0}dx_{n}\right),
	\end{split}
	\end{equation}
where $C_{1}$ is  a uniform constant.	
	
	
	Similarly with \eqref{stable 1 3}, we have
	\begin{equation}\label{stable 2 3}
	\begin{split}
	&\int_{\T_{p}^{2}}\left\{\big|\nabla_{\frac{\partial }{\partial x_{n}}}V\big|^{2}_{g}-R\left(\frac{\partial f_{p}}{\partial x_{n}}, V, \frac{\partial f_{p}}{\partial x_{n}}, V\right)\right\}dx_{0}dx_{n}\\
	=&\int_{\T_{p}^{2}}\left\{\big|V_{4}\big|_{g}^{2}+2\sum_{i, j\neq 0,n; i\neq j}g\left(\frac{\partial }{\partial x_{i}},\nabla_{\frac{\partial }{\partial x_{n}}}
	\frac{\partial}{\partial x_{j}}\right)\int_{\T_{p}^{2}}a_j\frac{\partial a_{i}}{\partial x_{n}}\,dx_{0}dx_{n}\right\}dx_{0}dx_{n}.
	\end{split}
	\end{equation}
	When $|w'|$ is small enough, as argued above, we have
	\begin{equation*}
	\left|\int_{\T_{p}^{2}}a_j\frac{\partial a_{i} }{\partial x_{n}}g\left(\frac{\partial }{\partial x_{i}},\nabla_{\frac{\partial }{\partial x_{n}}}\frac{\partial}{\partial x_{j}}\right)dx_{0}dx_{n}\right|
	\leq \frac{1}{4}\left(\int_{\T_{p}^{2}}\big|V_{2}\big|_{g}^{2}dx_{0}dx_{n}+\int_{\T_{p}^{2}}\big|V_{4}\big|_{g}^{2}dx_{0}dx_{n}\right).
	\end{equation*}
	Combining with \eqref{stable 1 3}, \eqref{SV 2 3} and \eqref{stable 2 3},  we conclude $I(V, V)\geq 0$. Hence $f_p$ is a stable harmonic map when $p$ is in some neighbourhood of $T_{1}$. Similarly,  $f_p$ is a stable harmonic map when $p$ is in some neighbourhood of $T_{i}$, if we consider $V_{(i-1)0}$, $i=1,\cdots, n$.
\end{proof}
\begin{theo}\label{unstable p n}
	The harmonic map $f_{p}$ is unstable, when $x_{1}=0$, $|x_{n+1}|>\hspace{-.15cm}\sqrt{2}$ and $w_{2}=\cdots=w_{n-1}=0$.
\end{theo}
\begin{proof}
	Choose $V=a_{1}\frac{\partial}{\partial x_{1}}+a_{n+1}\frac{\partial}{\partial x_{n+1}}$ with $a_{1}=\frac{1}{2\pi}\cos(2\pi x_{0}), a_{n+1}=\frac{1}{2\pi}\sin(2\pi x_{0})$. By \eqref{SV2 3}, Proposition \ref{tool} and $g_{1(n+1)}=0$, when $x_{1}=0$ and $w_{2}=\cdots=w_{n-1}=0$,
	\begin{equation*}
	\begin{split}
	&\big|\nabla_{\frac{\partial }{\partial x_{0}}}V\big|^{2}_{g}-R\left(\frac{\partial f_{p}}{\partial x_{0}}, V, \frac{\partial f_{p}}{\partial x_{0}}, V\right)\\
	=& \sin^2(2\pi x_0) \,g_{11}+\cos^2(2\pi x_0) \,g_{(n+1)(n+1)}-\frac{1}{2\pi}\left(g_{01,n+1}-g_{0(n+1),1}\right)\ \ \ \ \ \ \ \\
	=&\frac{1}{2}\left(\frac{1}{A}-\frac{3|w_{1}|^{2}}{4A^{2}}+\frac{\epsilon^{2}|w_{1}|^{2}}{4A^{2}}\right)-\left(1+\epsilon^{2}\right)\frac{1}{4A^{2}}x_{n+1}^{2}
	< -\frac{1}{4A}+\frac{3}{4A^{2}}<0,
	\end{split}
	\end{equation*}
	when $|x_{n+1}|>\hspace{-.15cm}\sqrt{2}$. 
		Note $\frac{\partial a_{1}}{\partial x_{n}}=\frac{\partial a_{n+1}}{\partial x_{n}}=0$. Hence  $V_{4}=0$.
	By \eqref{stable 2 3},
	\begin{equation*}
	\int_{\T_{p}^{2}}\left\{|\nabla_{\frac{\partial }{\partial x_{n}}}V|^{2}_{g}-R\left(\frac{\partial f_{p}}{\partial x_{n}}, V, \frac{\partial f_{p}}{\partial x_{n}}, V\right)\right\}dx_{0}dx_{n}=0.
	\end{equation*}
	It follows $I(W, W)<0.$
\end{proof}

\begin{theo}\label{minimal surface n}
	Each fibre $\T^2_p$ of the Hopf fibration is a flat and totally geodesic minimal surface in $(\S^{2n-1}\times\S^1,{g}_\epsilon)$. It is a stable minimal surface if $p=p_0$ and it is unstable if $p$ is as in Theorem \ref{unstable p n}, where $p_0\in T_{1}.$
\end{theo}
\begin{proof}
	As argued for Theorem \ref{minimal surface}, $\T^2_p$ is totally geodesic in $(\S^{2n-1}\times\S^1,{g}_\epsilon)$ and is an unstable minimal surface if $f_p$ is an unstable harmonic map.

	We now consider $p_{0}\in T_{1}$. 
	Similarly with \eqref{curvature of torus 1}, we have, at $w'=0$, i.e., at $p_0$,
	\begin{equation}\label{curvature of torus 1 n}
	R_{0i0j}=0, \ \ 	R_{0inj}=0, \ \ 	R_{ninj}=0,\ \ 0\leq i,j\leq 2n-1.
	\end{equation}
Using the argument of Theorem \ref{minimal surface}, 	 $\T^{2}_{p_0}$ is a stable minimal surface. \end{proof}

\section{Appendix}

The following result mentioned Introduction should be well known. A proof may not be explicitly documented in the literature, we include one below.
\begin{prop}
There is no stable branched minimal immersion of a compact Riemann surface $\Sigma$ in $\S^3\times\S^1$ equipped with the product metric $g=g_1\oplus g_2$ where $g_1$ is the constant curvature 1 metric on $\S^3$ and $g_2$ is any Riemannian metric on $\S^1$.
\end{prop}
\begin{proof}
Suppose $f:\Sigma\to\S^3\times\S^1$ is a branched conformal map that is harmonic as well. We write $f = (f_1,f_2)$ where $f_1:\Sigma\to\S^3,f_2:\Sigma\to\S^1$. Due to the product structure of the metric, $f_1$ is a harmonic map into the round $3$-sphere $\S^3$ (similarly for $f_2$). By a result of Leung \cite{Le} the harmonic map $f_1$ must be unstable or constant. The latter is impossible since a nonconstant conformal map cannot have image in $\{q\} \times \S^1$ for some $q\in\S^3$.

There is a variation $V_1$ of the unstable harmonic map $f_1$ making second variation of the energy negative. In other words, there is a family of $f^t_1\in C^2(\Sigma,\S^3)$ parametrized by $t$ with $f^0_1=f_1$, so that $V_1=\frac{\partial f^t_1}{\partial t}\big|_{t=0}$, and
$$
\int_{\Sigma}\left[\big|\nabla^{g_1}_{\frac{\partial}{\partial \overline{w}_{0}}}V_1\big|^2_{g_1}-\mbox{tr}_{g_1} R_{g_1}\left(V_1,\frac{\partial f_{1}}{\partial w_{0}},V_1, \overline{\frac{\partial f_{1}}{\partial w_{0}}}\right)\right]d\mu_{T^2}<0
$$
where $\nabla^{g_1}$ is the pullback connection on the bundle $f^*_1T\S^3$ over $\Sigma$ by $f_1$.

Set $V=(V_1,0)$. For $f^t=(f^t_1,f_2)$, it is clear that $f^0=f$ and $V=\frac{\partial f^t}{\partial t}\big|_{t=0}$. As $g=g_1\oplus g_2$, the pullback connection
$\nabla^g$ on the pullback bundle $f^*T(\S^3\times\S^1)$ over $T^2$ by $f$ splits into $\nabla^{g_1}+\nabla^{g_2}$, and the Riemannian curvature simplifies according to
$$
R_g(X,Y ,X,Y)=R_{g_1}(X_1,Y_1, X_1,Y_1)+R_{g_2}(X_2,Y_2, X_2,Y_2)=R_{g_1}(X_1,Y_1,X_1,Y_1 )
$$
 where $X=X_1+X_2,Y=Y_1+Y_2$ and $X_1,Y_1\in T\S^3,X_2,Y_2\in T\S^1$. Then
\begin{align*}
\int_{\Sigma}&\left[\big|\nabla^{g}_{\frac{\partial}{\partial \overline{w}_{0}}}V\big|^2_g - \mbox{tr}_gR_g\left(V,\frac{\partial f}{\partial w_{0}},V,\overline{\frac{\partial f}{\partial w_{0}}}\right)\right]d\mu_{\Sigma}\\
&=\int_{\Sigma}\left[  \big|\nabla^{g_1}_{\frac{\partial}{\partial \overline{w}_{0}}}V_1\big|^2_{g_1}-\mbox{tr}_{g_1}R_{g_1}\left(V_1,\frac{\partial f_{1}}{\partial w_{0}},V_1,\overline{\frac{\partial f_{1}}{\partial w_{0}}}\right)   \right] d\mu_{\Sigma}<0.
\end{align*}
Therefore $f$ is unstable as a harmonic map. In fact, the argument up to this point holds for harmonic maps from any compact manifold to $\S^3\times\S^1$ with the product metric $g$; so they are all unstable. 

Now, since $i_E>0$, the inequality $i_E\leq i_A$ in \cite{EM} implies that $f$ is unstable for area as a minimal surface.
\end{proof}

\end{document}